\documentclass[12pt]{amsart}
\usepackage{amsmath, amssymb, amsthm}
\usepackage{geometry}
\usepackage{graphicx}
\usepackage{tikz}
\usepackage{physics}
\newtheorem{thm}{Theorem}[section]

\newtheorem{prop}[thm]{Proposition}

\newtheorem{lem}[thm]{Lemma}
\newtheorem{defn}[thm]{Definition}

\newtheorem{assumption}[thm]{Assumption}

\newtheorem{remark}{Remark}
\usepackage{placeins}
\usepackage{subcaption}
\usepackage{float}
\usepackage{xcolor}
\usetikzlibrary{calc} 
\tikzset{my dbl/.style={double,double distance=2pt}}

\usepackage{comment}
\usepackage[colorlinks=true, linkcolor=blue]{hyperref}
\usetikzlibrary{arrows.meta,decorations.pathmorphing}
\usetikzlibrary{arrows.meta, positioning}
\usepackage{tabularx}
\nocite{*}

\title{Universality of AMP via Tree Pairings}
\author{David Kogan}
\address{Department of Mathematics, Yale University, New Haven, CT 06511, USA}
\email{d.kogan@yale.edu}
\date{September 2025}

\begin{document}

\begin{abstract}
We prove universality for Approximate Message Passing (AMP) with polynomial nonlinearities applied to symmetric sub-Gaussian matrices $A\in\mathbb R^{N\times N}$.  
Our approach is combinatorial: we represent AMP iterates as sums over trees and define a Wick pairing algebra that counts the number of valid row-wise pairings of edges. The number of such pairings coincides with the trees contribution to the state evolution formulas.
This algebra works for non-Gaussian entries.  
For polynomial nonlinearities of degree at most $D$, we show that the moments of AMP iterates match their state evolution predictions for $t \lesssim \frac{\log N}{D\log D}$ iterations. The proof controls all ``excess'' trees via explicit enumeration bounds, showing non ``Wick-paired" contributions vanish in the large-$N$ limit.  
The same framework should apply, with some modifications, to spiked AMP and tensor AMP models.
\end{abstract}
\maketitle

\section{Introduction}

Approximate Message Passing (AMP) algorithms are iterative methods for solving 
high-dimensional inference problems involving random matrices. 
They were introduced in \cite{donoho2009message}.
In certain random settings, their performance can be predicted exactly by a simple 
deterministic formula called \emph{state evolution} (SE), first proved for Gaussian 
matrices by \cite{bayati2011dynamics, bolthausen2014iterative} using the Gaussian conditioning technique. The results have been extended to more general algorithms \cite{chen2021universality, montanari2021estimationlowrankmatricesapproximate, dudeja2023universality, wang2024universality} and to algorithms where the number of iterations $t$ is allowed to grow with the sample size $N$, \cite{rush2018finite, li2023approximate, han2025entrywise}.

The combinatorial approach to state evolution has been studied in
\cite{bayati2015universality, jones2024fourieranalysisiterativealgorithms}, in which AMP algorithms are analyzed 
via tree-structured diagrams representing the moments of the iterates. This approach is similar to the moment method for studying the eigenvalues of random matrices \cite{furedi1981eigenvalues,bai1993limit}.
In this work, we continue in this combinatorial direction and develop a 
\emph{Wick pairing algebra} on rooted trees that captures the moment expansions of AMP 
in both the unspiked and spiked Wigner models. In contrast to \cite{jones2024fourieranalysisiterativealgorithms}, we choose to work with the standard monomial basis rather than the Hermite basis, as it allows us to control the error terms with greater precision, which is relevant when the number of iterations grows with $N$. Using the combinatorial approach, we are able to show that state evolution formulas hold for polynomial nonlinearities of degree at most $D$ for $t\lesssim \log N/(D\log D)$ iterations. $D$ is allowed to depend on $N$, and the result works for sub-Gaussian matrices.

We show that, for any fixed number of iterations and polynomial nonlinearities, 
the Wick moment functional coincides with the state evolution prediction on the 
$\Delta=0$ trees, i.e., those in which every edge label $(i,j)$ appears exactly 
twice and the two copies occur in the same generation, giving a perfect row-wise 
Wick pairing. The main advantage of this approach is that it allows one to diagrammatically track the contributions of individual terms in the moment expansion, organized by 
their underlying tree structure. The tree moment algebra allows one to see the exact contribution from any tree structured moment in the algorithm to the state evolution formula as $N\to\infty$. There is no assumption of rotation invariance on the underlying matrix ensemble, so the analysis applies to general sub-Gaussian models beyond the Gaussian orthogonally invariant case. To show that the approach works for $t\lesssim \log N/(D\log D)$ requires a careful analysis of error terms arising from $\Delta>0$ (non-perfectly paired) trees. The key idea is that we can group trees based on their non-perfectly Wick paired subtrees. We remark that bound $t\lesssim \log N/(D\log D)$ is sharp for this setting and we construct counterexamples to the state evolution predictions when $t$ exceeds this threshold, i.e., when $\log N/(D\log D)\lesssim t\lesssim \log N/\log D$. (See section \ref{sec:counter_example})

The Wick pairing algebra developed here is not restricted to the symmetric Wigner case.  
In particular, the same combinatorial structure should naturally extend to spiked AMP models, rank-$R$ AMP, where the signal component is a sum of $R$ rank-one terms, and to tensor AMP models. This can be done by tracking whether each edge comes from the signal or noise matrix.

\subsection*{Contributions}

\noindent\textbf{Organization.}
Section~\ref{sec:tree-expansion} develops the tree expansion and $\Delta$ parameter.
Section~\ref{sec:wick-algebra} defines the Wick pairing algebra and proves its equivalence
to state evolution formulas for finite $t$. Section~\ref{sec:universality} gives the universality proof for
polynomial AMP for $t\lesssim \log N/(D\log D)$ iterations. Section~\ref{sec:spiked} explains how the results could extend to a spiked AMP model, although the proof is not provided.

\smallskip

We now introduce the setting of our analysis. 
Following \cite{bayati2015universality}, we have the following definition:

\begin{defn}[AMP Instance] \label{defn:amp_instance}
An \textbf{AMP instance} is the data tuple $(A, \mathcal F, x^0, D)_N$ where:
\begin{itemize}
    \item[(1)] $A \in \mathbb{R}^{N \times N}$ is a symmetric matrix with zero diagonal, i.e., $A_{ii} = 0$ for all $i \in [N]$. The off-diagonal entries $A_{ij}$ are independent, mean-zero sub-Gaussian random variables satisfying
    \[
    \log \mathbb{E}[e^{\lambda A_{ij}}] \leq \frac{C\lambda^2}{2N^2}, \quad \text{and} \quad \mathbb{E}[A_{ij}^2] = \frac{1}{N}.
    \]
    $C$ is independent of $N$
    \item[(2)] $\mathcal {F} = \{f_t(\cdot) :\ t \in \mathbb{Z}_{\ge 0} \}$ is a collection of time-indexed polynomial functions, where each $f_t(\cdot): \mathbb{R} \to \mathbb{R}$ is a polynomial of degree at most $D$. Assume $f_0(x)=x$.

\item[(3)] The initialization \(x^0=(x_i^0)_{1\le i\le N}\) has independent sub-Gaussian entries,
independent of \(A\), variance \(\tau_0\), and moments satisfying
\[
\mathbb{E}|x_i^0|^{\,p}\le (Kp)^{p/2}.
\]
Note that we can initialize $x^0=\textbf {1}_N$ (i.e., the vector of all ones in $\mathbb R^N$).
\end{itemize}
\end{defn}

Let $(A, \mathcal F, x^0, D)_N$ be an AMP instance. Then we define the AMP iterates recursively by,
\[
x^{t+1}_i = \sum_{j\in [N]}A_{ij} f_t(x^t_j) - \sum_{j=1}^N A_{ij}^2f'(x_j^t)f_{t-1}(x_i^{t-1})\]
With the convention that $f_{-1}\equiv 0$. The appearance of the second term called the ``Onsager Term" cancels certain correlations between iterates. The behavior of the AMP algorithm can be predicted by a deterministic scalar recursion called the state evolution (SE). In the limit as $N\to\infty$ the finite moments of the iterates $x_i^t$ agree with the finite moments of a Gaussian with variance $\tau_t^2$, where the sequence $\{\tau_t^2\}$ can be defined recursively by
\[\tau_0 = \mathbb E(x_i^0)^2,\quad  \tau_{t+1}^2 = \mathbb{E}[f_t(\tau_tZ)^2], \quad Z\sim \mathcal{N}(0, 1)\]

In addition to proving universality, for fixed iterations, we also analyze iterations on the order of $t=O(\log N)$. In this setting, we need an additional assumption.

\begin{assumption} \label{assump:bound}
For a given horizon \(t\), assume that all state evolution variances 
\(\{\tau_s^2\}_{s=0}^t\) and polynomial coefficients are uniformly bounded: there 
exists a constant \(M\ge 1\), independent of \(N\) and \(D\), such that
\[
\tau_s^2 \le M,\qquad |c_{s,d}|\le M
\]
for every \(s\in\{0,\dots,t\}\) and monomial coefficient \(c_{s,d}\) of \(f_s\) where $f_s(z)=\sum_{d=0}^Dc_{s, d}z^d$.

\noindent
This assumption is satisfied, for example, if we rescale the polynomials \(f_s\) so 
that the state evolution variances remain \(O(1)\).
\end{assumption}

The following is our main theorem.

\begin{thm}\label{thm:State_evolution_approximation_polynomial}

Consider an AMP instance $(A, \mathcal F, x^0, D)_N$ as in Definition \ref{defn:amp_instance}. Let \(\{\tau_t^2\}\) be the state-evolution variance sequence defined by $
\tau_0^2 = \mathbb{E}\big[(x_i^0)^2\big], 
\tau_{t+1}^2 = \mathbb{E}[f_t(\tau_t Z)^2], Z\sim\mathcal{N}(0,1)$.

Fix a sufficiently large constant \(K_1>20\), and define
\[
C_D := K_1D\log(\max\{D,M\}).
\]
Suppose Assumption~\ref{assump:bound} holds. Then, for all integers $t \le \frac{\log N}{C_D}$
, every \(i\in[N]\), and every fixed integer \(1\le m\le \log\log N\) and sufficiently large $N$,
\[
\left|\mathbb{E}\big[(x_i^t)^{m}\big] - \mathbb{E}\big[(y_i^t)^{m}\big]\right|
\le N^{-1/2 + B/K_1},
\]
where \(y_i^t\sim\mathcal{N}(0,\tau_t^2)\). Here $B$ is some absolute constant. Note that when $m$ is odd $\mathbb{E}\big[(y_i^t)^{m}\big]=0$.
\end{thm}

\begin{thm}
\label{thm:convergence_amp_probability}
Fix $t\ge 0$ and $m\ge 1$ in the regime of Theorem~\ref{thm:State_evolution_approximation_polynomial}. Define
\[
M_{t,m}^{(N)} := \frac1N\sum_{i=1}^N (x_i^t)^m,
\qquad
\mu_{t,m} := \mathbb E[y_t^m],\quad y_t\sim\mathcal N(0,\tau_t^2).
\]
Then $M_{t,m}^{(N)} \to \mu_{t,m}$ in probability as $N\to\infty$.
\end{thm}

The Theorem establishes the convergence of the moments of AMP iterates to the predictions from state evolution. In particular, it shows that for $t\leq (\log N)/C_D$ iterations, the algorithm behaves asymptotically Gaussian. Moreover, the polynomial degree $D$ is allowed to grow with $N$; however, in this case, the constant $C_D$ will also grow with $N$. We remark that the proof is not sharp regarding the constant $B$, but it is smaller than $1000$.
\section*{Acknowledgements}
The author is grateful to Zhou Fan for suggesting the problem and for many helpful discussions and insights.

\section{Tree Expansion of AMP Iterates}
\label{sec:tree-expansion}
In \cite{bayati2015universality}, a graphical representation of the AMP iterates via trees was introduced to encode the monomial structure of the iterates. When the functions \( f_t \) are polynomials, each coordinate \( x_i^t \) becomes a polynomial in the matrix entries of \( A \) and the entries of the initialization \( x^0 \). We represent such expansions using trees. We start with a simple illustrative example of the tree expansion of the iterates in AMP.

\subsection*{Example: Tree Expansion}

Let the AMP iteration use a linear nonlinearity at time \( t = 0 \), and a cubic nonlinearity at time \( t = 1 \):
\[
f_0(z) = z, \quad f_1(z) = z^3.
\]
We analyze the AMP iterate \( x_i^2 \). The iteration gives:
\begin{align*}
x_j^1 &= \sum_{k=1}^N A_{jk} x_k^0, \\
x_i^2 &= \sum_{j=1}^N A_{ij} \left( x_j^1 \right)^3 - \sum_{j=1}^N A_{ij}^2\cdot 3(x_j^1)^2\cdot x_i^0
\end{align*}

Substitute \( x_j^1 = \sum_{k} A_{jk} x_k^0 \) into the expression for \( x_i^2 \). We obtain:
\[
x_i^2 = \sum_{j=1}^N A_{ij} \left( \sum_{k=1}^N A_{jk} x_k^0 \right)^3 - \sum_{j=1}^N A_{ij}^2\cdot 3(x_j^1)^2\cdot x_i^0
\]

Expanding the expression gives:
\[
x_i^2 = \sum_{j=1}^N \sum_{k_1,k_2,k_3=1}^N A_{ij} A_{jk_1} A_{jk_2} A_{jk_3} x_{k_1}^0 x_{k_2}^0 x_{k_3}^0 - \sum_{j=1}^N A_{ij}^2\cdot \left(3\sum_{k_1, k_2=1}^N A_{jk_1}A_{jk_2} x_{k_1}^0x_{k_2}^0\right)\cdot x_i^0
\]
Now consider the following tree that represents $x_i^2$:
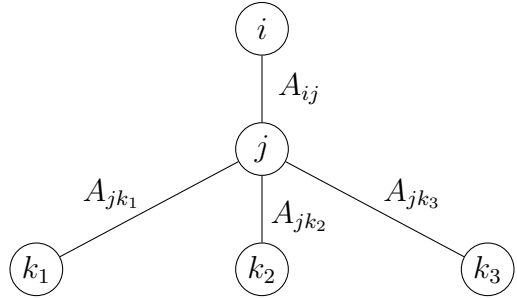
\begin{center}
\begin{tikzpicture}[
  level distance=1.6cm,
  every node/.style={circle, draw, minimum size=7mm, inner sep=0pt},
  level 1/.style={sibling distance=3cm},
  level 2/.style={sibling distance=3cm}
]
\node {$i$}
  child {node {$j$}
    child {node {$k_1$}}
    child {node {$k_2$}}
    child {node {$k_3$}}
  };

\node[draw=none, fill=none] at (0.5,-0.8) {$A_{ij}$};
\node[draw=none, fill=none] at (-2,-2.2) { $A_{jk_1}$};
\node[draw=none, fill=none] at (0.5,-2.5) {$A_{jk_2}$};
\node[draw=none, fill=none] at (2,-2.2) { $A_{jk_3}$};
\end{tikzpicture}
\end{center}

Each edge in the tree corresponds to a matrix entry of $A$ indexed by the labels of the vertices that are incident to the edge. We take the product over all edges and all leaf vertices, which contribute a factor $x^0_i$ for each leaf vertex with label $i$. If we sum over all tree labelings with integers in $\{1, \dots, N\}$ of the $4$ non-root vertices, we obtain precisely the term $\sum_{j=1}^N \sum_{k_1,k_2,k_3=1}^N A_{ij} A_{jk_1} A_{jk_2} A_{jk_3} x_{k_1}^0 x_{k_2}^0 x_{k_3}^0$.

The Onsager Term can be interpreted as canceling all ``backtracking trees." These are trees in which there exists a sequence of three vertices $u, v, w$ with an edge between $(u, v)$ and $(v, w)$ with $u, v, w$ at three different levels such that $u$ and $w$ have the same label. In our example, the term $\sum_{j=1}^N A_{ij}^2\cdot 3(x_j^1)^2\cdot x_i^0$ exactly cancels those trees where either $k_1=i$, $k_2=i$, or $k_3=i$.

\subsection{Monomial Trees}

We first consider the simplified case when each function $f_t$ is a monomial.

\begin{defn}
Let $(A, \mathcal{F}, x^0, D)_N$ be an AMP instance where each \( f_s \in \mathcal {F} \) is a monomial of the form \( f_s(z) = c_s z^{d_s} \). A \emph{labeled monomial tree} of depth \( t \) is a rooted, labeled tree with the following properties:
\begin{itemize}
    \item The tree has \( t + 1 \) levels of vertices, corresponding to time steps \( 0 \) through \( t \), with the root at level \( t \) and  the leaves at level \( 0 \). There are $t$ edges between the root and the leaves.
    \item The root has $m$ children.
    \item Each vertex is labeled by an index in \( [N] \), representing a coordinate of the AMP iterate.
    \item Each vertex at level \( s \in \{1, \dots, t\} \) has \( d_s \) children, where \( d_s \) is the degree of \( f_s \).
    \item Each edge corresponds to a matrix entry \( A_{ij} \), where the edge connects a parent vertex with label \( i \) to a child vertex with label \( j \). (Note: we will later require that for any edge in a monomial tree $\mathcal T$ between vertices with labels $(i, j)$, there is at least one other edge in the tree between the labels $(i, j)$)
    \item The tree satisfies the \emph{non-backtracking condition}: for any path \( i \to j \to k \) between three distinct vertex levels, we require \( i \neq k \).
\end{itemize}
We will refer to trees without vertex labels as unlabeled monomial trees.
\end{defn}

Now consider the case \( f_t(z) = \sum_{j=0}^D c_{t,d} z^d \) for a general polynomial. In that case, the number of children at each vertex may vary depending on which monomial term contributes. Let $\mathcal{T}_{i,t, m}$ denote the class of all labeled monomial trees that are obtained by selecting a monomial $c_{t, d}z^d$ at each vertex and considering all such possible choices and all possible labelings. We also require that the root has $m$ children and has the label $i$. Then we have the following proposition: 

\begin{lem}\label{lem:tree_expansion}
Let \( f_t(z) = \sum_{j=0}^D c_{t,d} z^d \) and suppose the AMP instance is as above. Then, for each \( i \in [N] \) and each time \( t \geq 0 \), the AMP iterate \( x_i^t \) can be written as
\[
(x_i^t)^{m} = \sum_{\mathcal{T} \in \mathcal{T}_{i,t, m}} \operatorname{Val}(\mathcal{T}),
\]
where each \( \mathcal{T} \in \mathcal{T}_{i,t, m} \) is a labeled monomial tree of depth \( t \) rooted at vertex \( i \) such that the root has $2m$ children, and \( \operatorname{Val}(\mathcal{T}) \) is given by
\[
\operatorname{Val}(\mathcal{T}) = \prod_{v \in V(\mathcal{T})} c_v \cdot \prod_{(u, v) \in E(\mathcal{T})} A_{\ell(u) \ell(v)} \cdot \prod_{\ell \in L(\mathcal{T})} x^0_\ell.
\]

Where \( E(\mathcal T) \) is the set of edges in the tree \( \mathcal T \), \(V(\mathcal T)\) is the set of vertices in $\mathcal T$, and \(L(\mathcal T)\) is the set of leaf vertices in $\mathcal T$. The coefficients $c_v$ are given by the coefficient of the monomial applied at the vertex $v$. Here $\ell(u)$ and $\ell(v)$ are the labels of the vertices $u$ and $v$.
\end{lem}

The correctness of this expansion has been proven in \cite{bayati2015universality} by induction on the depth of the tree. Note that we will sometimes need the value of a labeled forest $\mathcal F$ consisting of several labeled trees $\mathcal T_1 ,\dots \mathcal T_k$. In this case, we write $\mathrm{Val}(\mathcal F)=\prod_{j=1}^k\mathrm{Val}(\mathcal T_j)$.

\subsection{Some combinatorial results on trees}
We will now study the combinatorics of the monomial trees in the expansion of the AMP iterates. For this, we have the following definition:

\begin{defn}
    Let $\mathcal{T}$ be a labeled monomial tree associated with an AMP instance of depth $t$, where each vertex is labeled by an index in $[N]$. Two such trees $\mathcal{T}_1$ and $\mathcal{T}_2$ are said to be \textbf{isomorphic} if there exists a permutation $\pi \in S_N$ such that applying $\pi$ to every vertex label of $\mathcal{T}_1$ yields $\mathcal{T}_2$ note that the two trees must have the same underlying tree structure (but possibly different labelings of vertices).

    We define a \textbf{tree isomorphism class} to be the equivalence class of labeled monomial trees under this relation. That is, two trees belong to the same class if they are identical up to a relabeling of vertex indices by an element of $S_N$.
\end{defn}

\begin{remark}
Given an unlabeled tree $T$ (possibly with the root labeled), we may describe an 
isomorphism class of labeled trees by specifying which vertices of $T$ 
are identified as having the same label. This identification determines 
the class uniquely, without reference to any particular choice of labels from $[N]$.  

Formally, an isomorphism class corresponds to a partition of the vertex set 
$V(T)$ in which vertices in the same block are assigned the same label.  
The labeled monomial trees belonging to this class are precisely the 
injective labelings of the blocks of this partition by distinct elements of $[N]$ 
(except that the root label is typically fixed in our setting).  
\end{remark}

\begin{assumption}
    We will consider monomial trees with the following property. For any edge in a monomial tree $\mathcal T$ between vertices with labels $(i, j)$ there is at least one other edge in the tree between the labels $(i, j)$. If there are indices $(i, j)$ with only one edge between that pair, then $\mathbb E\textrm{Val}(\mathcal T)=0$ since we can factor $A_{ij}$ out of the expectation and $\mathbb EA_{ij}=0$. Such trees will not contribute to moment calculations.
\end{assumption}

Throughout, we use the following convention:
\begin{itemize}
  \item \(\mathcal T\) or \(\mathcal T_\phi\) for a labeled monomial tree. $\phi$ denotes a labeling of the vertices of an unlabeled tree.
  \item \(T\) for an unlabeled tree shape.
  \item \(\mathcal{I}\) for an isomorphism class of labeled trees arising from a 
        fixed unlabeled tree.
  \item \(T_\mathcal I\) to denote the underlying unlabeled tree of an isomorphism class $\mathcal I$.
\end{itemize}

\begin{defn}
\label{def:excess}
Let \( \mathcal{T} \) be a labeled monomial tree. The \textbf{excess} of \( \mathcal{T} \) is defined as
\[
\Delta(\mathcal{T}) :=  \frac{1}{2} |E(\mathcal{T})|-|V(\mathcal{T})| + 1,
\]
where \( |V(\mathcal{T})| \) denotes the number of distinct vertex labels, and \( |E(\mathcal{T})| \) denotes the number of edges in the tree.

Note that $\Delta(\mathcal{T})$ is fixed on each tree isomorphism class $\mathcal I$, so this definition is well-defined for both individual monomial trees and their isomorphism classes. Throughout, we use the notation \( \Delta(\mathcal{T}) \) or $ \Delta(\mathcal{I}) $ to refer to the excess of either a labeled tree or an isomorphism class. Also note that $+1$ in this definition comes from fixing the root label.
\end{defn}

\begin{defn}
\label{def:standard_enumeration}
Let \( \mathcal{T} \) be a labeled monomial tree of depth \( t \), and suppose the tree is organized into generations (levels) \( 0, 1, \dots, t \), where the root lies at level \( t \) and the leaves at level \( 0 \). For each level \( s = 1, \dots, t \), let \( n(s) \) denote the number of edges in generation \( s \) (i.e., edges connecting vertices at level \( s \) to level \( s-1 \)). Label the edges in generation \( s \) as
\(
e_1(s), e_2(s), \dots, e_{n(s)}(s)
\) such that for any \( i < j \), the parent of edge \( e_i(s) \) appears earlier (i.e., has a smaller index) than the parent of \( e_j(s) \).

The \textbf{standard ordering of edges} in \( \mathcal{T} \) is then the sequence:
\[
e_1(1), e_2(1), \dots, e_{n(1)}(1),\ e_1(2), \dots, e_{n(2)}(2),\ \dots,\ e_1(t), \dots, e_{n(t)}(t),
\]
that is, edges are ordered generation-by-generation from bottom to top, and within each generation according to the rule above.
\end{defn}

In this section, we will prove several results about labeled monomial trees; however, there are analogous results for the isomorphism classes of those trees. Similar to \cite{bai1993limit}, we define several types of edges.

\begin{defn}
\label{def:edge_types}
Following the standard ordering of the edges in a labeled monomial tree \( \mathcal{T} \) (as in Definition~\ref{def:standard_enumeration}), we classify each edge into one of three types based on the occurrence of its vertex label pair. Assume all edges are directed toward the root. Let each edge be represented by an ordered pair of vertex labels \( (i, j) \), where \( i \) is the parent and \( j \) is the child.

\begin{itemize}
    \item A \textbf{\( \mathbf{T_1} \)} edge is one for which the vertex label \( j \) has not appeared previously in the tree (i.e., this is the first appearance of the label \( j \) in the standard ordering).

    \item A \textbf{\( \mathbf{T_2} \)} edge is the second occurrence of an edge between vertex labels \( (i, j) \), where the first such occurrence was a \( \mathbf{T_1} \) edge.

    \item A \textbf{\( \mathbf{T_3} \)} edge is any remaining edge that is neither \( \mathbf{T_1} \) nor \( \mathbf{T_2} \).
\end{itemize}
\end{defn}

\begin{defn} \label{defn:gb_pair}
    Consider a pair of $\bf{T_1}$ and $\bf{T_2}$ edges between the vertices $(v_1, u_1)$ and $(v_2, u_2)$, respectively. Then, by Definition \ref{def:edge_types}, we assume $v_1$ and $v_2$ have the same label $i$, and $u_1$ and $u_2$ have the same label $j$. We call this pair of edges a \textbf{good pair} if both pairs $v_1, v_2$ and $u_1, u_2$ are in the same generation. Otherwise, we call this pair a $\textbf{bad pair}$.
\end{defn}

\begin{lem} \label{lem:excess_positive_T}
    For any labeled monomial tree $\Delta(\mathcal T)\geq 0$.
\end{lem}

\begin{proof}
    We claim that each non-root vertex label must appear at least twice. Suppose there is some label $i$ of a vertex $v$ that appears only once. Let $u$ be the parent vertex of $v$ with label $j$. Then we claim that the edge between labels $(i, j)$ appears only once. Indeed, since $i$ only appears once as a label, the only other possible location for an edge between labels $(i, j)$ would be between vertex $v$ and some other vertex $w$. Then the vertices $u, v, w$ violate the backtracking condition, which is a contradiction. Therefore, we conclude that $|V(\mathcal T)|\leq |E(\mathcal T)|/2+1$.
\end{proof}

\begin{lem}\label{lem:t4_bound}
    The number of $\bf{T_3}$ edges in a labeled monomial tree $\mathcal T$ is bounded by $2\Delta$. 
\end{lem}

\begin{proof}
    Consider the subtree $\mathcal T'\subset \mathcal T$ obtained by keeping only $\bf{T_1}$ and $\bf{T_2}$ edges and identifying vertices with the same label. It satisfies $|V(\mathcal T')|=|E(\mathcal T')|/2+1$ since $\mathcal T'$ can be viewed as a tree with edges appearing twice ($|V(\mathcal T')|, |E(\mathcal T')|$ denotes the number of distinct vertex labels and edges in $\mathcal T'$). Then, since the $\bf{T_3}$ edges do not introduce new vertex labels, the number of $\bf{T_3}$ edges is bounded by $2\Delta$.
\end{proof}

\begin{lem}\label{lem:Ni_3}
    Let $N_i$ denote the number of occurrences of label $i$ in a labeled monomial tree $\mathcal T$.  Then $\sum_{i:N_i\geq 3}N_i\leq 6\Delta(\mathcal T)$. Let $M_r$ denote the number of occurrences of the root label, other than the root vertex, then $M_r\leq 2\Delta(\mathcal T)$.
\end{lem}

\begin{proof}
    Follow the standard ordering. Let $|U(\mathcal T)|$ denote the total number of vertices in $\mathcal T$ excluding the root, and $|V(\mathcal T)|$ denote the number of distinct vertex labels in $\mathcal T$. Observe $|E(\mathcal T)|=|U(\mathcal T)|$ where $|E(\mathcal T)|$ is the number of edges in $\mathcal T$. Then $|U(\mathcal T)|=\sum_{i:N_i=2} N_i+\sum_{i:N_i\geq 3}N_i+M_r$ (where we sum over labels $i$ that are not equal to the root label) and $|V(\mathcal T)|=\sum_i 1$. Then $\Delta(\mathcal T)\geq |U(\mathcal T)|/2-|V(\mathcal T)|=\sum_{i:N_i\geq 3}N_i/2-\sum_{i:N_i\geq 3}1$ since 
    $6(\sum_{i:N_i\geq 3}N_i/2-\sum_{i:N_i\geq 3}1)\geq \sum_{i:N_i\geq 3}N_i$ the first result follows. The second result follows similarly.
\end{proof}

\begin{lem} \label{lem:bij_bound}
    Let $b_{ij}$ denote the number of edges between vertices with labels $i$ and $j$. Then $\sum_{b_{ij}:b_{ij}>2}b_{ij}\leq 6\Delta$
\end{lem}

\begin{proof}
    For any value $b_{ij}>2$, there must be at least one $\bf{T_3}$ edge between vertices with labels $i$ and $j$. Furthermore, only two edges between vertices with labels $i$ and $j$ are not $\bf{T_3}$. Hence, by Lemma \ref{lem:t4_bound}, we conclude that $\sum_{b_{ij}:b_{ij}>2}b_{ij}\leq 6\Delta$.
\end{proof}

\subsection{Combinatorics of Wick Pairings}
In this section, we describe how the state evolution formula arises from precisely the labeled monomial trees with $\Delta(\mathcal T)=0$. We will use the notation $\ell(v)$ to denote the label of the vertex $v$, and $d(v)$ to denote the distance of the vertex $v$ from the root (the number of edges).

\begin{prop}\label{prop:bad_pair_bound_delta0}
     There are no bad pairs in a tree labeled monomial tree $\mathcal T$ if $\Delta(\mathcal T)=0$.
\end{prop}
\begin{proof}
    Let $e_1=(v_1, u_1)$ and $e_2=(v_2, u_2)$ be a bad pair of edges. By the definition of a bad pair, both edge endpoints are label-equivalent: that is,
    \(
    \ell(v_1) = \ell(v_2) \text{ and } \ell(u_1) = \ell(u_2)
    \). Since the pair is bad, at least one of the following holds, $d(v_1)\ne d(v_2)$ or $d(u_1)\ne d(u_2)$. Where $d$ denotes the distance to the root. Assume WLOG $d(v_1)<d(v_2)$. Let $p^1(e_1)$ denote the parent edge of $e_1$, and $p^k(e_1)$ denote the edge $p^1(p^1(p^1\dots p^1(e_1)))$ where we take the parent $k$ times. Let $p^{k_0}(e_1)$ denote the final parent edge when we reach the root. By Lemma \ref{lem:t4_bound} none of the edges $p^0(e_1), p^1(e_1), \dots, p^{k_0}(e_1)$ are a $\bf{T_3}$ edge, and by Lemma \ref{lem:Ni_3} $M_r=0$, there are no labels $i$ such that $N_i\geq 3$.
    
      Now, suppose $p^0(e_1), p^1(e_1), \dots, p^{k_0}(e_1)$ consists of all $\bf{T_1}$ and $\bf{T_2}$ edges. Then the edges $p^0(e_1), p^1(e_1), \dots, p^{k_0}(e_1)$ form a path that terminates in the root. Each of the edges $p^0(e_1), p^1(e_1), \dots, p^{k_0}(e_1)$ has a corresponding pair $p^0(e_1)^*, p^1(e_1)^*, \dots, p^{k_0}(e_1)^*$, i.e. the corresponding $\bf{T_1}$ or $\bf{T_2}$ edge. If these pairs do not form a single path, then for some $i$ we have $N_i\geq 3$. For example, take $p^0(e_1)$ and $p^0(e_1)^*$. We claim $p^1(e_1)^*$ shares a vertex with $p^0(e_1)^*$. Since $p^0(e_1)$ and $p^1(e_1)$ share a vertex label $j$ the edges $p^0(e_1)^*$ and $p^1(e_1)^*$ also contain the label $j$. Therefore, if $p^0(e_1)^*$ and $p^1(e_1)^*$ do not share a vertex with label $j$ then $N_j\geq 3$. Now, since $d(v_1)<d(v_2)$ and $p^0(e_1)^*, p^1(e_1)^*, \dots, p^{k_0}(e_1)^*$ form a path that terminates in a vertex with the same label as the root, this implies $M_r\geq 1$, which is a contradiction. We conclude that there are no bad pairs.
\end{proof}

\begin{lem}\label{lem:delta_0_perfect_pairing}
    Let $\mathcal{T}$ be a labeled monomial tree. Then $\Delta(\mathcal{T}) = 0$ if and only if $\mathcal{T}$ admits a perfect ``Wick pairing" of edges at each level. That is, the label of any non-root vertex appears exactly twice, and both occurrences are in the same generation. This produces a pairing of edges, row by row. Each pair consists of one $\bf{T_1}$ and one $\bf{T_2}$ edge.
    
      In particular, a Wick Pairing is one with no bad pairs and no $\bf{T_3}$ edges. Note that the Lemma also holds for an isomorphism class of labeled trees.
\end{lem}

\begin{proof}
    First, suppose that $\Delta(\mathcal T)=0$. By Lemma \ref{prop:bad_pair_bound_delta0}, there are no bad pairs. There are either $0$ or two edges between any pair of vertex labels $(i, j)$, since there are no $\bf{T_3}$ edges, by Lemma \ref{lem:t4_bound}. Thus, we see that each edge is paired with exactly one other edge in the same level, which proves the claim. By definition, each pair must consist of one $\bf{T_1}$ and one $\bf{T_2}$ edge.

    In the other direction, suppose $\mathcal T$ admits a perfect Wick Pairing. Then, the resulting graph can be viewed as a tree with doubled edges. Then we see that $|E(\mathcal T)|/2+1=|V(\mathcal T)|$
    
      The same proof also applies to tree isomorphism classes.
\end{proof}

\subsection{Wick's Pairings for Monomial Trees}

\begin{prop} \label{prop:tree_pairs_delta_0}
    Let \( f_t(z) = z^{d_t} \) and fix a positive integer \( m \). Then the number of isomorphism classes of labeled monomial trees \( \mathcal{T} \) with excess \( \Delta(\mathcal{T}) = 0 \) contributing to \( \mathbb{E}[(x_i^t)^{2m}] \) is
\[
(2m - 1)!! \cdot (2d_{t-1} - 1)!!^m \cdot (2d_{t-2} - 1)!!^{md_{t-1}} \cdots (2d_1 - 1)!!^{md_{t-1} \cdots d_2}.
\]
\end{prop}
\begin{proof}
    By Lemma \ref{lem:delta_0_perfect_pairing}, a tree $\mathcal T$ (or tree isomorphism class) with \( \Delta(\mathcal T) = 0 \) admits a perfect ``Wick Pairing". Thus, the contributions to \( \mathbb{E}[(x_i^t)^{2m}] \) arise only from trees in which all edge labels can be perfectly Wick-paired at each level. Now we go row by row and count the number of ways to pair the edges. The top row has $2m$ edges, which contribute the factor $(2m-1)!!$. In the second row, there are now $m$ pairs of vertices, which have a total of $2d_{t-1}$ children per pair. We can then pair those children in $(2d_{t-1}-1)!!$ for a total factor of $(2d_{t-1}-1)!!^m$ for the second row. Continuing in this manner, we can obtain the formula.
\end{proof}

Note that if the exponent is 
odd ($2m+1$), no Wick pairings exist since there are an odd number of edges in the first row.

\begin{prop}
    The state evolution variance agrees with Proposition \ref{prop:tree_pairs_delta_0} when $x^0=\bf 1$ (i.e. the vector of all $1$'s in $\mathbb R^N$).
\end{prop}
\begin{proof}
    Again, suppose that $f_t(z)=z^{d_t}$. Then the SE formula indicates that the variance at iteration $t$ satisfies $\tau_t^2=\mathbb E[f_{t-1}(\tau_{t-1}Z)^2]$ where $Z\sim\mathcal N(0,1)$. We induct on $t$. Suppose $x^0$ is deterministic (this can be modified). By definition, we have that $\tau_0=1$, $\tau_1=1$; therefore, $\tau_2=(2d-1)!!$. 

    For the inductive step, we have that $\tau_t^2=\mathbb E[f_{t-1}(\tau_{t-1}Z)^2]=\mathbb E[(\tau_{t-1}Z)^{2d_{t-1}}]=\tau_{t-1}^{2d_{t-1}}\mathbb E[Z^{2d_{t-1}}]$, where $\tau_{t-1}^2=(2d_{t-2}-1)!!(2d_{t-3}-1)!!^{d_{t-2}}\dots (2d_1-1)!!^{d_{t-2}\dots d_2}$, so $$\tau_t^2=(2d_{t-1}-1)!!(2d_{t-2}-1)!!^{d_{t-1}}(2d_{t-3}-1)!!^{d_{t-1}d_{t-2}}\dots (2d_1-1)!!^{d_{t-1}d_{t-2}\dots d_2}$$
\end{proof}

\section{AMP Moment Algebra}
\label{sec:wick-algebra}

Lemma \ref{lem:delta_0_perfect_pairing} states that each labeled tree with excess $\Delta=0$ has its edges perfectly Wick paired row by row. Given an unlabeled tree $T$, the following definition counts the number of ways to produce such a perfect pairing from $T$. In other words, how many isomorphism classes with excess $\Delta=0$ can be produced from $T$.

\begin{defn}[Wick Pairing] \label{defn:tree_wick_pairings}
Let \(T\) be an unlabeled tree whose vertices are organized by depth.
For each level \(\ell \ge 0\), let \(V_\ell\) denote the set of vertices at
depth \(\ell\), and let
\[
E_\ell = \{ e = (u, v) : u \in V_\ell,\ v \in V_{\ell+1} \}
\]
be the set of edges between levels \(\ell\) and \(\ell+1\).

A \emph{row-wise Wick pairing} of \(T\) is a family \(\{\pi_\ell\}_{\ell \ge 0}\)
such that, for every \(\ell\),

\begin{itemize}
    \item \(\pi_\ell\) is a partition of \(V_\ell\) into unordered pairs
    \(\{v,v'\}\) such that the parents $p(v), p(v')\in V_{\ell-1}$ of $v,v'$ are either the same or belong to the same pair in $\pi_{\ell-1}$
    \item In particular, such a family $\{\pi_\ell\}_{\ell\geq 0}$ produces a perfect matching on the edge sets \(E_\ell\). For each edge $(p(v), v)$, it is paired with $(p(v'),v')$, where $\{v, v'\}$ forms a block in $\pi_{\ell+1}$.
\end{itemize}

In particular, a Wick pairing exists only if \(|E_\ell|\) is even for every
\(\ell\). We define
\[
\mathrm{Wick}(T)
\]
to be the number of such row-wise Wick pairings of \(T\), i.e.\ the number of
ways to pair all edges between consecutive levels. Equivalently, a Wick
pairing specifies, at each level, which vertices are identified as having the same label, without reference to any
specific numeric labels. This results in an isomorphism class $\mathcal I$, and we require $\Delta(\mathcal I)=0$. See Figure \ref{fig:wick_pairings}.

Suppose $T_1, \dots, T_r$ are unlabeled trees, except for the root, and they have the same root label. We may define $\mathrm{Wick}(T_1, \dots, T_r)=\mathrm{Wick}(T_1\star\dots\star T_r)$ where $\star$ denotes the operation of gluing two trees at their common root label. 
\end{defn}

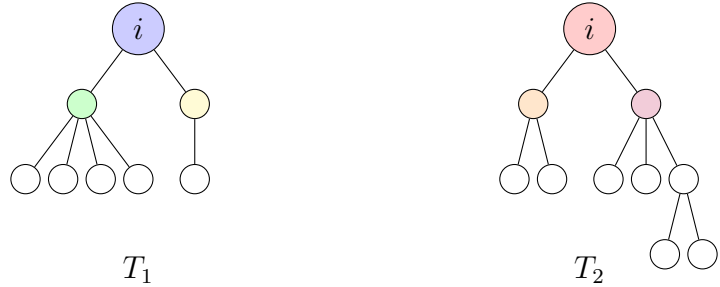
\begin{figure}[h]
\centering
\begin{tikzpicture}[level distance=1cm,
  level 1/.style={sibling distance=1.5cm},
  level 2/.style={sibling distance=0.5cm}]

\node[circle,draw,fill=blue!20] (r1) at (0,0) {\(i\)}
  child {node[circle,draw,fill=green!20] (a1) {}
    child {node[circle,draw,fill=white] (a11) {}}
    child {node[circle,draw,fill=white] (a12) {}}
    child {node[circle,draw,fill=white] (a13) {}}
    child {node[circle,draw,fill=white] (a14) {}}
  }
  child {node[circle,draw,fill=yellow!20] (a2) {}
    child {node[circle,draw,fill=white] (a21) {}}
  };

\node[circle,draw,fill=red!20] (r2) at (6,0) {\(i\)}
  child {node[circle,draw,fill=orange!20] (b1) {}
    child {node[circle,draw,fill=white] (b11) {}}
    child {node[circle,draw,fill=white] (b12) {}}
  }
  child {node[circle,draw,fill=purple!20] (b2) {}
    child {node[circle,draw,fill=white] (b21) {}}
    child {node[circle,draw,fill=white] (b22) {}}
    child {node[circle,draw,fill=white] (b23) {}
    child {node[circle,draw,fill=white] (b41) {}}
    child {node[circle,draw,fill=white] (b41) {}}}
  };

\node at (0,-3.2) {\( {T}_1 \)};
\node at (6,-3.2) {\( {T}_2 \)};

\end{tikzpicture}
\caption{In the example above there are $45$ Wick pairing between the trees \( {T}_1 \) and \( {T}_2 \) with root label $i$. After gluing the trees at the root we note that green and orange vertices must be paired and the yellow and purple vertices must be paired (by parity considerations). Then there are $5\cdot 3\cdot 1$ ways to pair the children of green and orange and $3\cdot 1$ ways to pair the children of yellow and purple for a total of $45$ pairings, so $\mathrm{Wick}(T_1, T_2)=45$. The bottom two vertices of tree $T_2$ must always be paired together.}
\label{fig:wick_pairings}
\end{figure}

\begin{defn}[AMP Moment Algebra on Unlabeled Trees] \label{defn:tree_algebra}
Let \( \mathcal A_i \) denote the set of all unlabeled trees in which only a root is labeled \( i \in [N] \), and all other vertices are unlabeled. Each tree is a monomial tree: every vertex \( v \) has an out-degree \( d_v \in \mathbb{N} \), and edges connect vertices from depth \( s \) to \( s+1 \). In particular, each tree is graded by vertex depth from its root. Let
\[
\mathcal A := \bigsqcup_{i=1}^N \mathcal A_i.
\]

The \emph{AMP Moment Algebra on Rooted Trees} is the triple
\[
(\mathcal{A}, \star, \langle \cdot, \cdot \rangle),
\]
defined as follows:

\begin{enumerate}
    \item \textbf{Vector space:}  
    \[
    \mathcal{A} := \operatorname{span}_{\mathbb{R}} \{ T : T \in \mathcal A \}, \quad \mathcal{A}_i := \operatorname{span}_{\mathbb{R}} \{ T \in \mathcal A_i \}.
    \]

    \item \textbf{Product} \( \star : \mathcal{A}_i \times \mathcal{A}_i \to \mathcal{A}_i \):  
    For \( T_1, T_2 \in \mathcal A_i \), define \( T_1 \star T_2 \) to be the rooted tree formed by taking the disjoint union of \( T_1 \) and \( T_2 \) and identifying their labeled roots into a single root labeled \( i \). The new root inherits all the children of both original roots. Extend \( \star \) bilinearly. For \( i \ne j \), set \(  T_1 \star  T_2 := 0 \).

    \item \textbf{Inner product:}  
    For \( T_1, T_2 \in \mathcal A_i \), define
    \[
    \langle  T_1,  T_2 \rangle := \operatorname{Wick}( T_1,  T_2),
    \]
    where \( \operatorname{Wick}( T_1,  T_2) \) is the set of valid Wick pairings (Definition~\ref{defn:tree_wick_pairings}) on basis trees, and the inner product is extended bilinearly to $\mathcal A$.

    \item \textbf{Moment functional:}  
    We extend $\mathrm{Wick}(T)$ linearly to $\mathcal A$.
\end{enumerate}

Then it is not hard to see that \( (\mathcal{A}_i, \star, \langle \cdot, \cdot \rangle) \) is a commutative algebra. We remark that the inner product commutes with $\star$, i.e., $\mathrm{Wick}( T_1\star  T_2)=\mathrm{Wick}( T_1,  T_2)$ simply by the definition.
\end{defn}

\begin{lem}
    $\langle  T_1,  T_2\rangle
      :=\bigl|\mathrm{Wick}( T_1, T_2)\bigr|$ is a positive definite inner product on $\mathcal A_i$.
\end{lem}

\begin{proof}
    It is simple to see that for any unlabeled basis tree $T_1$, we have $\langle  T_1,  T_1\rangle>0$ for any nonempty tree since overlaying the tree on top of itself gives a valid Wick pairing of $T_1$ with itself.
\end{proof}

A key observation about the AMP Moment Algebra is that the functional 
\(\mathrm{Wick}(T)\) coincides with the prediction of state 
evolution (SE) for the variance of an AMP iterate whose combinatorial expansion 
is given by the rooted tree \(T\). The algebra structure $(\star)$ is primarily about 
gluing trees along their roots, and this operation naturally extends the 
Wick pairing rule to compute not only variances but also all mixed moments of AMP 
iterates.

\begin{prop} \label{prop:se_functional_tree}
Let \( T \in \mathcal A_i \) be a rooted unlabeled tree. Define the SE functional \( \mathcal{L}( T) \) recursively as follows:

\begin{itemize}
    \item If \(  T \) is an edge tree (i.e. consists of one edge), then \( \mathcal{L}(\mathcal T) := 1 \).
    \item If the root of \(T\) has children \(  T^{(1)}, \dots,  T^{(d)} \), then
    \[
    \mathcal{L}( T) := \mathbb{E}[Z_1 \cdots Z_d], \quad \text{where } (Z_1, \dots, Z_d) \sim \mathcal{N}(0, \Sigma),
    \]
    and the covariance matrix \( \Sigma \in \mathbb{R}^{d \times d} \) is given by
    \[
    \Sigma_{ij} :=\mathcal L( T^{(i)}\star  T^{(j)})
    \]
\end{itemize}
Then we have the equality \( \mathcal{L}(\mathcal F) = \mathrm{Wick}(\mathcal F)\).
\end{prop}

\begin{proof}
We induct on the number of edges in the rooted tree \(  T \in \mathcal A_i \). Suppose \(  T \) is an edge tree, consisting of a single edge connecting the root to a single leaf. Then, by definition, \( \mathcal{L}( T) := 1 \), and \( \mathrm{Wick}( T) = 1 \).

Assume that for all trees \(  T' \) with fewer than \( k \) edges, we have
\(
\mathcal{L}( T') = \mathrm{Wick}( T').
\)
Now let \( \ T \) be a tree with \( k \) edges. Suppose the root has children \(  T^{(1)}, \dots,  T^{(d)} \). By the inductive hypothesis, for each pair \( i, j \), we have:
\[
\Sigma_{ij} = \mathrm{Wick}( T^{(i)} \star \mathcal T^{(j)}).
\]

Now, by Wick's theorem for multivariate Gaussians, we have:
\[
\mathbb{E}[Z_1 \cdots Z_d] = \sum_{\text{pairings } \pi \text{ of } [d]} \prod_{\{i,j\} \in \pi} \Sigma_{ij},
\]
and thus,
\[
\mathcal{L}(\mathcal T) = \sum_{\text{pairings } \pi} \prod_{\{i,j\} \in \pi} \mathrm{Wick}(\mathcal T^{(i)} \star \mathcal T^{(j)}).
\]

But each such pairing corresponds to a Wick pairing of the full tree \(  T \). So, to construct a Wick pairing of \(  T \), we must Wick pair the subtrees \(  T^{(i)} \) and \(  T^{(j)} \) in pairs and then count the number of Wick contractions of the resulting glued trees. So the total number of valid Wick contractions is exactly:
\[
\mathrm{Wick}( T) = \sum_{\text{pairings } \pi} \prod_{\{i,j\} \in \pi} \mathrm{Wick}( T^{(i)} \star  T^{(j)}),
\]
matching the recursive computation of \( \mathcal{L}(T) \).

Hence, \( \mathcal{L}( T) = \mathrm{Wick}( T) \) for all trees \(  T \) with \( k \) edges, completing the inductive step.

\end{proof}

This result provides a concrete combinatorial interpretation of state evolution. Each tree in the forest represents a polynomial in the matrix entries of $A$ that appear in the expansion of an AMP iterate, and the total contribution to its variance or covariance is given by counting Wick pairings within each tree independently. We now prove the universality of the entire AMP moment Algebra for finite $D, t$.

\begin{thm}
\label{thm:universality_algebra_elements}

Fix an integer \(K<\infty\), and let \(\mathcal{A}_K\) denote the set of 
unlabeled trees (except for the root) with at most \(K\) edges. Let \( A = (A_{ij})_{1 \leq i, j \leq N} \) be a random matrix whose entries satisfy:
\[
\mathbb{E}[A_{ij}] = 0, \quad \mathbb{E}[A_{ij}^2] = \frac{1}{N}, \quad \text{and} \quad \log \mathbb{E}[e^{\lambda A_{ij}}] \leq \frac{C \lambda^2}{2N^2} \quad \forall \lambda \in \mathbb{R}.
\]
For any tree \( {T} \in \mathcal A_K \), define its value by:
\[
\operatorname{Val}({T}) := \sum_{\phi: V({T}) \to [N]} \prod_{(u,v) \in E(\mathcal{T})} A_{\phi(u)\phi(v)},
\]
where \( \phi \) is any labeling of the vertices of \( {T} \). Then as \( N \to \infty \),
\[
\mathbb{E}[\operatorname{Val}({T})] = \mathrm{Wick}({T}) + O_K(N^{-1/2}),
\]
where \( \mathrm{Wick}({T}) \) denotes the number of Wick pairings of \( {T} \), and the implicit constant in the error term depends only on \( K \) and $\lambda$, but not on $N$.
\end{thm}

\begin{proof}
Let \( \mathcal{T}_\phi \) denote the labeled tree obtained by assigning labels \( \phi: V({T}) \to [N] \). Decompose the expectation as a sum over isomorphism classes of such labeled trees.

Let us first consider labeled trees \( \mathcal{T}_\phi \) with excess \( \Delta(\mathcal{T}_\phi) = 0 \), i.e., each edge label \( (i,j) \) appears exactly twice in the product. By Lemma~\ref{lem:delta_0_perfect_pairing}, such labelings correspond exactly to Wick pairings, and the number of isomorphism classes of such trees is, by definition, \( \mathrm{Wick}({T}) \). We have,
\[
\mathbb{E}\left[\prod_{(u,v) \in E(\mathcal{T}_\phi)} A_{\phi(u)\phi(v)}\right] = N^{-|E({T})|/2},
\]
since each distinct label pair contributes \( \mathbb{E}[A_{ij}^2] = 1/N \). Let $|V(\mathcal T_\phi)|$ denote the number of distinct vertex labels in the labeling $\phi$. Then, for a fixed value $|V(\mathcal T_\phi)|$, the number of possible injective labelings $\phi$ (excluding the root) is 
\[
N(N-1)\cdots (N - |V(\mathcal{T}_\phi)| + 2) = N^{|V(\mathcal{T}_\phi)|-1} (1 + O(N^{-1})),
\]
and using \( \Delta(\mathcal T_\phi) = 0 \Rightarrow |E(T)| /2= |V(\mathcal{T}_\phi)| - 1 \), we find
\[
\mathrm{Wick}({T}) \cdot N^{|V(\mathcal{T}_\phi)|-1 - |E(T)|/2} (1 + O(N^{-1})) = \mathrm{Wick}(T)(1 + O(N^{-1})),
\]
since the exponent simplifies to 0.

Now consider labelings \( \phi \) with \( \Delta(\mathcal{T}_\phi) > 0 \). From the subgaussian moment bound:
\[
\mathbb{E}[|A_{ij}|^p] \leq C_1^p \cdot N^{-p} \cdot p^{p/2},
\]
and by Lemma~\ref{lem:bij_bound}, the total sum of all \( b_{ij} > 2 \) satisfies:
\[
\sum_{(i,j):\, b_{ij}>2} b_{ij} \leq 6\Delta.
\]
Thus, the full product is bounded by:
\[
\mathbb{E}\left[\prod_{(u,v) \in E(\mathcal{T}_\phi)} A_{\phi(u)\phi(v)}\right]
\leq C_2^{6\Delta} \cdot (6\Delta)^{3\Delta} \cdot N^{-|E(\mathcal{T})|/2}.
\]

There are at most \( K_\Delta <\infty\) (not dependent on $N$) isomorphism classes of labeled trees with excess \( \Delta \), and for each class \( \mathcal{I} \) with \( V(\mathcal{I}) \) vertices, there are at most \( N^{V(\mathcal{I})} \leq N^{|E(\mathcal{I})|/2 - \Delta + 1} \) injective labelings.

So the total contribution from all \( \Delta > 0 \) terms is bounded by:
\[
\sum_{\Delta \geq 1/2}^{|V(T)|} K_\Delta \cdot N^{-\Delta} \cdot C_2^{6\Delta} \cdot (6\Delta)^{3\Delta},
\]
which is \( O_K(N^{-1/2}) \). Summing both contributions gives:
\[
\mathbb{E}[\operatorname{Val}({T})] = \mathrm{Wick}({T}) + O_K(N^{-1/2}).
\]
\end{proof}

Theorem~\ref{thm:universality_algebra_elements} implies that, as \( N \to \infty \), the expectation functional \( \mathbb{E}[\operatorname{Val}(\cdot)] \) on the subspace \( \mathcal A_K \subset \mathcal A \) converges to the Wick pairing functional defined purely combinatorially on tree monomials. The universality also holds for finite linear combinations of such trees. Since the Wick pairing defines an inner product on trees, this yields a universal limiting law for all mixed moments of AMP iterates corresponding to polynomial functions with bounded degree and depth. The proof for a growing number of iterations will require a more careful analysis of the error terms in the algorithm.

\section{State Evolution for Symmetric AMP}
\label{sec:universality}

In this section, we prove Theorem \ref{thm:State_evolution_approximation_polynomial} and Theorem \ref{thm:convergence_amp_probability}. This states that the moments of AMP iterates agree with the state evolution predictions for $t\lesssim \log N/(D\log D)$ iterations. The argument has two components:

\begin{itemize}
    \item[$(1)$] the $\Delta = 0$ part: Wick-paired trees exactly reproduce
    the state evolution moments.

    \item[$(2)$] the $\Delta > 0$ part: all excess trees contribute only a
    negligible error.
\end{itemize}

We treat these two regimes separately.  The key point is that $\Delta=0$
trees have the ``Wick pairing'' structure (Lemma \ref{lem:delta_0_perfect_pairing}), while
$\Delta>0$ trees contain a ``bad subtree'' that includes label collisions and $\bf{T_3}$ edges. Excluding the bad subtree, the remainder of the trees are ``Wick paired"; thus, we bound the contribution by using results from the $\Delta=0$ case.

\subsection{State Evolution for \texorpdfstring{$\Delta=0$}{Delta=0} with Polynomials}

\begin{lem}
\label{lem:pairing_lemma}
Let \( v_1 \) and \( v_2 \) be two vertices in a tree diagram that share the same label \( i \), and assume that each has parent vertices \( w_1 \) and \( w_2 \), respectively. Suppose the edges \( (w_1, v_1) \) and \( (w_2, v_2) \) are paired in the Wick pairing.

Let \( f_{t_1}^i \) and \( f_{t_2}^i \) be the polynomial AMP functions applied at \( v_1 \) and \( v_2 \), respectively. Interpret each \( f_{t_j}^i \) as an element in the tree algebra \( \mathcal{A} \), represented as a linear combination of monomial trees of depth $1$ corresponding to the polynomial expansion of \( f_{t_j}^i \).

Then the total contribution of all valid Wick pairings between the subtrees rooted at \( v_1 \) and \( v_2 \), weighted by monomial coefficients, is equal to the inner product
\[
\langle f_{t_1}^i, f_{t_2}^i \rangle_{\mathcal{A}} = \mathbb{E}[f_{t_1}^i(Z) f_{t_2}^i(Z)],
\]
where \( Z \sim \mathcal{N}(0,1) \).
\end{lem}

\begin{proof}

The proof essentially follows from the linearity of the inner product on the moment algebra $\mathcal A$. Suppose $f_{t_1}^i(x) = c_{d_1}^1 x^{d_1} + c_{d_{1}-1}^1 x^{d_1-1} + \dots + c_1^1 x + c_0^1$ and $f_{t_2}^i(x) = c_{d_2}^2 x^{d_1} + c_{d_{2}-1}^2 x^{d_1-1} + \dots + c_1^2 x + c_0^2$. Then,
\begin{align*}
    \mathbb{E}[f_{t_1}^i(Z) f_{t_2}^i(Z)]=\sum_{\ell_1, \ell_2|\ell_1+\ell_2 \textrm{ even }}c^1_{\ell_1}c^2_{\ell_2}\cdot (2\ell_1+2\ell_2-1)!!
\end{align*}

The term $c^1_{\ell_1}c^2_{\ell_2}\cdot (2\ell_1+2\ell_2-1)!!$ counts the number of Wick pairings of the depth one tree associated with the terms $c^1_{\ell_1}x^{\ell_1}$ and $c^2_{\ell_2}x^{\ell_2}$, and thus summing over all trees, we obtain the desired result.
\end{proof}

\bigskip

In the following proposition, we will show that the number of isomorphism classes of trees with $\Delta=0$ (weighted by the appropriate coefficients) gives the correct moments so that $x^t_i\sim N(0,\tau_t^2)$. Now we allow the number of children to vary to account for the fact that $f_t(z)$ is a polynomial with terms of different degrees.
Recall from Lemma \ref{lem:tree_expansion}
\[
(x_i^t)^{m} = \sum_{\mathcal{T} \in \mathcal{T}_{i,t, m}} \operatorname{Val}(\mathcal{T}),
\]
where each \( \mathcal{T} \in \mathcal{T}_{i,t, m} \) is a labeled monomial tree of depth \( t \) rooted at vertex \( i \) such that the root has $m$ children, and \( \operatorname{Val}(\mathcal{T}) \) is given by
\[
\operatorname{Val}(\mathcal{T}) = \prod_{v \in V(\mathcal{T})} c_v \cdot \prod_{(u, v) \in E(\mathcal{T})} A_{uv} \cdot \prod_{\ell \in L(\mathcal{T})} x^0_\ell.
\]

\begin{prop}
\label{prop:SE_from_delta_zero_trees}
Let \( x_i^t \) denote the AMP iterate at time \( t \) with initialization \( x^0 \), and nonlinearity polynomials \(f_0, f_1, \dots, f_{t-1} \). Define \( {T}_{i,t,m} \) to be the set of depth \( t \) rooted unlabeled monomial trees contributing to the $2m$th moment \( \mathbb{E}[(x_i^t)^{2m}] \), where each vertex at depth \( s \) has an out-degree equal to the degree of a monomial in \( f_s \). Then for $t\geq 1$,
\begin{align}
\label{eq:se_pairings}
\tau_t^{2m}\cdot (2m-1)!! = \sum_{{T} \in {T}_{i,t,2m}} \mathrm{Wick}({T}) \cdot \prod_{v \in V({T})} c_v\cdot \prod_{\ell \in L({T})} \tau_0,
\end{align}
Where $V(T)$ and $L(T)$ denote the sets of vertices and leaf vertices, respectively.
\end{prop}

\begin{proof}
    Consider the case $s=1$. In this case, there is one tree with $2m$ children, since we assume $f_0(x)=x$. Then $\mathrm{Wick}(T)=(2m-1)!!$ and we get a factor of $\tau_0^{2m}$ from the leaves. Since $\tau_1=\mathbb E[Z^2]=1$, both sides agree.

    Assume the claim holds at depth \( s \), so that:
Now assume the claim holds for some depth \(s \ge 1\), so that
\[
\tau_s^{2m} \cdot (2m-1)!! \;=\;
\sum_{T \in T_{i,s,m}} \mathrm{Wick}(T) \cdot \prod_{v \in V(T)} c_v \cdot \prod_{\ell \in L(T)} \tau_0.
\]
We prove the statement for depth \(s+1\). We consider depth \( s + 1 \). We will count the number of ways to Wick pair the edges, starting with the unlabeled trees contributing to $(x_{s+1}^t)^{2m}$. First, pair the \( 2m \) edges in the first generation (distance 1 from the root). There are \( (2m - 1)!! \) such pairings.
For each pair of first-generation vertices \( (v_1, v_2) \), we apply Lemma~\ref{lem:pairing_lemma}. the total contribution from all ways of Wick-pairing the children of \( V_1 \) and \( V_2 \), weighted by the corresponding monomial coefficients, is:
\[
\mathbb{E}[f_{s}(Z)^2], \quad \text{where } Z \sim \mathcal{N}(0,1).
\]
This contributes a factor of $\mathbb E[f_s(Z)^2]^m$ since there are $m$ pairs. Now that we have paired the first two generations, by the inductive hypothesis, there are $\tau_s^2$ ways to pair the subtree rooted at each pair of second generation vertices. From this, we conclude that the total contribution from pairings is
\[
(2m-1)!!\cdot \mathbb E[f_s(\tau_sZ)^2]^m, \quad \text{where } Z \sim \mathcal{N}(0,1).
\]
This matches the state evolution recursion and completes the inductive step.
\end{proof}

\begin{remark}
Proposition~\ref{prop:SE_from_delta_zero_trees} establishes that, for any fixed depth $t$ 
and polynomial nonlinearities $f_0,\dots,f_{t-1}$, the variance predicted by the AMP 
Moment Algebra coincides exactly with the state evolution (SE) recursion. 
Since Theorem~\ref{thm:universality_algebra_elements} already shows universality for the Wick pairing 
functional on the algebra of rooted trees, it follows immediately that SE holds for any 
finite number of iterations $t$ and any choice of polynomial nonlinearities in the setting considered here. Note for odd moments $(x_i^t)^{2m+1}$ both sides of \eqref{eq:se_pairings} are $0$. In particular, since there is an odd number of edges incident to the root, the excess of any tree in the expansion is positive $(\Delta>0)$ by Lemma \ref{lem:delta_0_perfect_pairing}.
\end{remark}

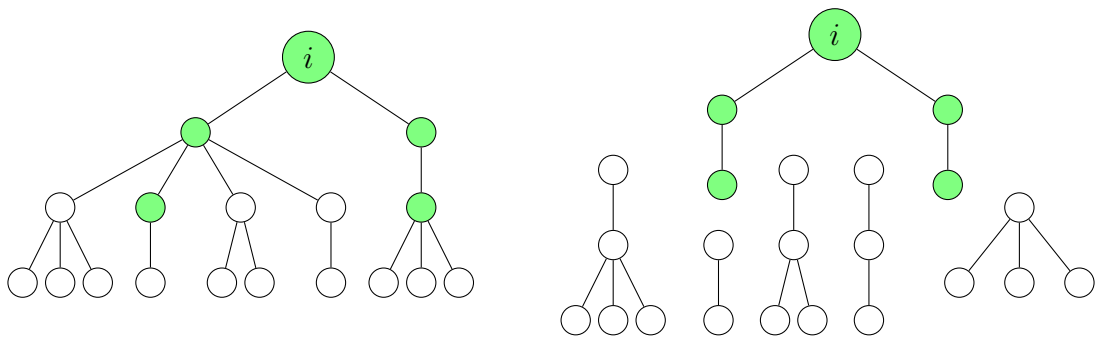
\begin{figure}[h]
\centering
\begin{tikzpicture}[level distance=1cm,
  level 1/.style={sibling distance=3cm},
  level 2/.style={sibling distance=1.2cm},
  level 3/.style={sibling distance=0.5cm}]

\node[circle,draw,fill=green!50] (r1) at (0,0) {\(i\)}
  child {node[circle,draw,fill=green!50] (a1) {}
    child {node[circle,draw,fill=white] (a11) {}
        child {node[circle,draw,fill=white] (a111) {}}
        child {node[circle,draw,fill=white] (a112) {}}
        child {node[circle,draw,fill=white] (a113) {}}
    }
    child {node[circle,draw,fill=green!50] (a12) {}
        child {node[circle,draw,fill=white] (a121) {}}}
    child {node[circle,draw,fill=white] (a13) {}
        child {node[circle,draw,fill=white] (a131) {}}
        child {node[circle,draw,fill=white] (a132) {}}}
    child {node[circle,draw,fill=white] (a14) {}
        child {node[circle,draw,fill=white] (a141) {}}}
  }
  child {node[circle,draw,fill=green!50] (a2) {}
    child {node[circle,draw,fill=green!50] (a21) {}
         child {node[circle,draw,fill=white] (a211) {}}
        child {node[circle,draw,fill=white] (a212) {}}
        child {node[circle,draw,fill=white] (a213) {}}}
  };

\node[circle,draw,fill=green!50] (r1) at (7,0.3) {\(i\)}
  child {node[circle,draw,fill=green!50] (a1) {}
    child {node[circle,draw,fill=green!50] (a21){}
  }
  }
  child {node[circle,draw,fill=green!50] (a1) {}
    child {node[circle,draw,fill=green!50] (a21){}
  }
  };

\node[circle,draw, right=0.25cm] (r1) at (9,-2) {}
  child {node[circle,draw,right=2cm] (a1) {}
  }
  child {node[circle,draw] (a1) {}
  }
  child {node[circle,draw, left=2cm] (a21){}
  };

  \node[circle,draw, right=0.25cm] (r1) at (3.6,-1.5) {}
  child {node[circle,draw] (a1) {}
      child {node[circle,draw, right=0.5cm] (a2) {}
      }
      child {node[circle,draw] (a2) {}
      }
      child {node[circle,draw, left=0.5cm] (a2) {}
      }
  };

\node[circle,draw, right=0.25cm] (r1) at (5,-2.5) {}
  child {node[circle,draw] (a1) {}
  };

\node[circle,draw, right=0.25cm] (r1) at (6,-1.5) {}
  child {node[circle,draw] (a1) {}
      child {node[circle,draw, right=0.15cm] (a2) {}
      }
      child {node[circle,draw, left=0.15cm] (a2) {}
      }
  };

\node[circle,draw, right=0.25cm] (r1) at (7,-1.5) {}
  child {node[circle,draw] (a1) {}
      child {node[circle,draw] (a2) {}
      }
  };
  
\end{tikzpicture}
\caption{A choice of a subtree (Green vertices) and the forest obtain by cutting along the boundary of the subtree}
\label{fig:tree_splitting}
\end{figure}

\FloatBarrier

\begin{lem} \label{lem:factor_val_forest}
    Let $\mathcal T$ be any labeled monomial tree, and let $\mathcal S\subset \mathcal T$ denote a subtree of $\mathcal T$ containing the root of $\mathcal T$. Let $\mathcal F$ be the forest consisting of $\mathcal S$ and the forest obtained by cutting $\mathcal S$ out from $\mathcal T$. That is, $\mathcal F$ consists of $\mathcal S$, and for each edge $e$ between a vertex of $\mathcal S$ on $\mathcal T\backslash\mathcal S$, there is a tree rooted at $e$. See Figure \ref{fig:tree_splitting}. Then $\mathcal F$ inherits labels from $\mathcal T$, and with this labeling, we have that 
    \[
    \mathrm{Val}(\mathcal T)= \prod_{\mathcal T_i\in \mathcal F}\mathrm{Val}(\mathcal T_i)
    \]
    That is, the RHS is a product over the trees in the forest $\mathcal F$.
\end{lem}

\begin{proof}
    The proof follows from the definition of $\mathrm{Val}$. Note that in $\prod_{\mathcal T_i\in \mathcal F}\mathrm{Val}(\mathcal T_i)$ we multiply by the coefficients $c_v$ for each vertex $v\in \mathcal T$ one time, meaning that we do not include the factor $c_{r_i}$ in $\mathrm{Val}(\mathcal T_i)$ where $r_i$ is the root vertex of the tree $\mathcal T_i\ne \mathcal S$. 
\end{proof}

Note that Proposition \ref{prop:SE_from_delta_zero_trees} considers unlabeled trees and their Wick pairings, where $\mathrm{Wick}(T)$ denotes the number of isomorphism classes we can produce by Wick pairing the unlabeled tree $T$. To extend \eqref{eq:se_pairings} to a sum over labeled trees, we will need the following Proposition. This will allow us to account for the error incurred by counting only injective labelings of isomorphism classes.
\begin{prop} \label{prop:factorial_formula}
    Let $(N)_k=N(N-1)\dots (N-k+1)$ and let $P_k$ denote the set of partitions on $k$ elements. Then the following equality holds.
    \[
    (N)_k=\sum_{\sigma\in P_k}(-1)^{k-|\sigma|}\left(\prod_{B\in \sigma}(|B|-1)!\right)N^{|\sigma|}
    \]
    Where $B$ denotes a block of $\sigma$ and $|\sigma|$ denotes the number of blocks in $\sigma$.
\end{prop}

\begin{proof}
    There is a standard formula that $(N)_k=\sum_{r=0}^ks(k, r)\cdot N^r$, where $s(k, r)$ denotes the Stirling numbers of the first kind. Here $s(k, r)$ denotes the number of permutations of $k$ elements with $r$ cycles times the sign $(-1)^{k-r}$. Given partition $\sigma\in P_k$, we can produce a permutation by ordering the elements in each block into a cyclic order. This can be done in $\prod_{B\in \sigma}(|B|-1)!$ ways, which results in $s(k, r)=(-1)^{k-r}\sum_{\sigma\in P_k:|\sigma|=r}\prod_{B\in \sigma}(|B|-1)!$. Plugging in the formula for $s(k, r)$ gives the desired formula.
\end{proof}

Currently, we have proven that,
\[
\tau_s^{2m} \cdot (2m-1)!! \;=\;
\sum_{T \in T_{i,s,2m}} \mathrm{Wick}(T) \cdot \prod_{v \in V(T)} c_v \cdot \prod_{\ell \in L(T)} \tau_0.
\]
Our goal will be to convert the RHS into a sum of the values of the labeled $\Delta=0$ trees, with an error term of $O(N^{-c})$ for a constant $c>0$.

Consider an isomorphism class $\mathcal I$ and write $\mathrm{Val}(\mathcal I)=\sum_{\mathcal T_\phi\in \mathcal I}\mathrm{Val}(\mathcal T_\phi)$; i.e., the value of an isomorphism class is the sum of the values of the trees contained in that isomorphism class. Let $V(\mathcal I)$ denote the set of non-root vertices of $\mathcal I$ where we view $\mathcal I$ as a quotient of $T_\mathcal I$ modulo the vertex label relation. The number of isomorphism classes of $\Delta=0$ graphs contained in a single unlabeled tree $T$ is $\mathrm{Wick}(T)$; thus, we need to show that \[\mathbb E\mathrm{Val}(\mathcal I)\approx\prod_{v \in V(T)} c_v \cdot \prod_{\ell \in L(T)}\tau_0\]
For an isomorphism class $\mathcal I$ produced by pairing the vertices of an unlabeled tree $T$.

\begin{defn} \label{defn:val_2_value}
    Let $\mathcal I$ be an isomorphism class produced from the unlabeled tree $T$ with $\Delta(\mathcal I)=0$. Consider the sum $\overline{\mathrm{Val}}(\mathcal I)$ defined as follows:
    \[\overline{\mathrm{Val}}(\mathcal I)=N^{-|V(\mathcal I)|}\sum_{\phi:V(\mathcal I)\to [N]} \prod_{v \in V(T)}  c_v \cdot \prod_{\ell \in L(T)}\tau_0\]
    Where on the RHS, we sum over all labelings $\phi:V(\mathcal I)\to [N]$ (not necessarily injective) of the vertices of $\mathcal I$. Note that equivalently, \[\overline{\mathrm{Val}}(\mathcal I) = \prod_{v \in V(T)} c_v \cdot \prod_{\ell \in L(T)}\tau_0\]
    Since the total number of possible labelings is precisely $N^{|V(\mathcal I)|}$.
\end{defn}

\begin{lem}
    \label{lem:delta_0_val_2}
    Let $\mathcal I_{i, s, 2m,0}$ denote the set of isomorphism classes of rooted trees with $\Delta(\mathcal I)=0$ in the expansion of the iterate $(x_i^t)^{2m}$. The following equality holds,
    \[
    \sum_{\mathcal I\in \mathcal I_{i, s, 2m,0}}\overline{\mathrm{Val}}(\mathcal I)=\tau_t^{2m}\cdot(2m-1)!!
    \]
\end{lem}

\begin{proof}
    For each unlabeled tree $T\in T_{i, t, 2m}$ in the expansion of $(x_i^t)^{2m}$, there are $\mathrm{Wick}(T)$ isomorphism classes that can be produced by pairing the vertices of $T$. By Definition \ref{defn:val_2_value}, we have that $\overline{\mathrm{Val}}(\mathcal I) = \prod_{v \in V(T)} c_v \cdot \prod_{\ell \in L(T)}\tau_0$. Hence, by Proposition \ref{prop:SE_from_delta_zero_trees} the claim holds.
\end{proof}

\begin{lem}
\label{lem:non_injective}
     Let $V(\mathcal I)$ denote the set of non-root vertices of an isomorphism class $\mathcal I$. We use the notation $T_{\mathcal I}$ to denote the underlying unlabeled tree that produced the isomorphism class $\mathcal I$. We have the following,
    \begin{align} \label{eq:difference_injective}
    \sum_{\mathcal I\in \mathcal I_{i, s, m,0}}\left(\overline{\mathrm{Val}}(\mathcal I)-\mathbb E\mathrm{Val}(\mathcal I)\right)=\sum_{\mathcal I\in \mathcal I_{i, s, m,0}}\sum_{\substack{\phi:V(\mathcal I)\to [N]\\ \phi \textrm{ non-injective}}} N^{-|V(\mathcal I)|} \prod_{v \in V(T_\mathcal I)}  c_v \cdot \prod_{\ell \in L(T_\mathcal I)}\tau_0
    \end{align}
\end{lem}

\begin{proof}
    Recall that for a labeled tree,
    \begin{align*}
    \operatorname{Val}(\mathcal{T}_\phi) = \prod_{v \in V(\mathcal{T})} c_v \cdot \prod_{(u, v) \in E(\mathcal{T})} A_{\phi(u)\phi(v)} \cdot \prod_{\ell \in L(\mathcal{T})} x^0_\ell.
    \end{align*}
    For the isomorphism class $\mathcal I$, we have that,
    \begin{align}
    \label{eq:val_I_injective}
    \mathbb E\mathrm{Val}(\mathcal I)=\sum_{\mathcal T_\phi\in \mathcal I}\mathbb E\mathrm{Val}(\mathcal T_\phi) = \sum_{\substack{\phi:V(\mathcal I)\to [N]\\ \phi \textrm{ injective}}} \mathbb E\left[\prod_{(u, v)\in E(T_\mathcal I)}A_{\phi(u)\phi(v)} \prod_{v \in V(T_\mathcal I)}  c_v \cdot \prod_{\ell \in L(T_\mathcal I)}x_{\phi(\ell)}^0\right]
    \end{align}
    Now observe that for each injective labeling $\phi$, the factor  $\mathbb E\prod_{(u,v)\in E(T_\mathcal I)} A_{\phi(u)\phi(v)}$ equals $N^{|V(\mathcal I)|}$ since each edge appears twice in a $\Delta=0$ labeled tree or isomorphism class, and the number of edges in $T_{\mathcal I}$ is twice the number of vertices in $V(\mathcal I)$ (recall $\mathbb EA_{ij}^2=1/N$). We can replace $\prod_{\ell \in L(T_\mathcal I)}x_{\phi(\ell)}^0$ with $\prod_{\ell \in L(T_\mathcal I)}\tau_0$, since $\mathbb E(x_i^0)^2=\tau_0$, the $x^0$ are independent of $A$, and leaf vertices come in pairs since $\Delta=0$. 
    Therefore, the RHS of \eqref{eq:val_I_injective} becomes,
    \begin{align*}   \sum_{\substack{\phi:V(\mathcal I)\to [N]\\ \phi \textrm{ injective}}} N^{-|V(\mathcal I)|} \prod_{v \in V(T_\mathcal I)}  c_v \cdot \prod_{\ell \in L(T_\mathcal I)}\tau_0
    \end{align*}
    Therefore, taking the difference as in \eqref{eq:difference_injective} for each isomorphism class $\mathcal I$, all that remains are the non-injective labelings, which proves the claim.
\end{proof}

It remains to bound the RHS of \eqref{eq:difference_injective}.

\begin{defn} \label{defn:subtree_sigma}
    Let $\mathcal I$, $(\Delta(\mathcal I)=0)$ be a tree isomorphism class with $|V(\mathcal I)|=k$ vertices and let $\sigma$ be a partition of $P_k$. Let $\sigma(\mathcal I)$ denote the subtree of $\mathcal I$ constructed in the following way: 
    \begin{itemize}
        \item[$(i)$] For each vertex $v\in \mathcal I$ such that $v$ is in a block of $\sigma$ with more than one element add $v$ to $\sigma(\mathcal I)$ and add all siblings of $v$ in $\mathcal I$ to $\sigma(\mathcal I)$.
        \item[$(ii)$] For each vertex added to $\sigma(\mathcal I)$ in step $(i)$, add the branch $b(v)$ that connects $v$ to the root (branches may overlap) and all siblings of the vertices on the branch $b(v)$.
    \end{itemize}
    See Figure \ref{fig:tree_subtree}.
\end{defn}

\begin{figure}[h]
\centering
\begin{tikzpicture}[level distance=1cm,
  level 1/.style={sibling distance=3cm},
  level 2/.style={sibling distance=1.2cm},
  level 3/.style={sibling distance=0.6cm}
  ]
\tikzset{
  small node/.style={
    circle,
    draw,
    minimum size=15pt,   
    inner sep=0pt,      
    font=\scriptsize,   
  }
}

\tikzset{
  rail/.style={
    draw=white,            
    line width=1.5pt,      
    double=black,          
    double distance=2pt, 
    line cap=butt,         
    double cap=butt,       
    shorten <=0.6pt,       
    shorten >=0.6pt        
  },
  every child/.style={
    edge from parent path={[my dbl] (\tikzparentnode) -- (\tikzchildnode)}
  }
}

\node[circle,draw,fill=white] (r1) at (0,0) {1}
  child {node[circle,draw,fill=white, minimum size=15pt, inner sep=0pt] (a1) {2}
    child {node[circle,draw,fill=red!50, minimum size=15pt, inner sep=0pt] (a11) {4}
        child {node[circle,draw,fill=green!50, minimum size=15pt, inner sep=0pt] (a111) {9}}
        child {node[circle,draw,fill=white, minimum size=15pt, inner sep=0pt] (a112) {10}}
        child {node[circle,draw,fill=white, minimum size=15pt, inner sep=0pt] (a113) {11}}
    }
    child {node[circle,draw,fill=green!50, minimum size=15pt, inner sep=0pt] (a12) {5}
        child {node[circle,draw,fill=white, minimum size=15pt, inner sep=0pt] (a121) {12}}}
    child {node[circle,draw,fill=white, minimum size=15pt, inner sep=0pt] (a13) {6}
        child {node[circle,draw,fill=green!50, minimum size=15pt, inner sep=0pt] (a131) {13}}
        child {node[circle,draw,fill=white, minimum size=15pt, inner sep=0pt] (a132) {14}}}
    child {node[circle,draw,fill=white, minimum size=15pt, inner sep=0pt] (a14) {7}
        child {node[circle,draw,fill=red!50, minimum size=15pt, inner sep=0pt] (a141) {15}}}
  }
  child {node[circle,draw,fill=white, minimum size=15pt, inner sep=0pt] (a2) {3}
    child {node[circle,draw,fill=white, minimum size=15pt, inner sep=0pt] (a21) {8}
         child {node[circle,draw,fill=white, minimum size=15pt, inner sep=0pt] (a211) {16}}
        child {node[circle,draw,fill=white, minimum size=15pt, inner sep=0pt] (a212) {17}}
    }
  };

\node[circle,draw,fill=white] (r1) at (7.5,0) {1}
  child {node[circle,draw,fill=white, minimum size=15pt, inner sep=0pt] (a1) {2}
    child {node[circle,draw,fill=red!50, minimum size=15pt, inner sep=0pt] (a11) {4}
        child {node[circle,draw,fill=green!50, minimum size=15pt, inner sep=0pt] (a111) {9}}
        child {node[circle,draw,fill=white, minimum size=15pt, inner sep=0pt] (a111) {10}}
        child {node[circle,draw,fill=white, minimum size=15pt, inner sep=0pt] (a111) {11}}
    }
    child {node[circle,draw,fill=green!50, minimum size=15pt, inner sep=0pt] (a12) {5}}
    child {node[circle,draw,fill=white, minimum size=15pt, inner sep=0pt] (a13) {6}
        child {node[circle,draw,fill=green!50, minimum size=15pt, inner sep=0pt] (a131) {13}}
        child {node[circle,draw,fill=white, minimum size=15pt, inner sep=0pt] (a131) {14}}}
    child {node[circle,draw,fill=white, minimum size=15pt, inner sep=0pt] (a14) {7}
        child {node[circle,draw,fill=red!50, minimum size=15pt, inner sep=0pt] (a141) {15}}}
  }
  child {node[circle,draw,fill=white, minimum size=15pt, inner sep=0pt] (a1) {3}};

\end{tikzpicture}
\caption{(Left) A tree isomorphism class $\mathcal I$ with $|V(\mathcal I)|=17$ vertices. Double edges denote the Wick pairing of an unlabeled tree. We choose the partition $\sigma=\{\{4, 15\}, \{9,5,13\},\dots\}\in P_{17}$ i.e. the white vertices are singleton blocks. (Right) The tree $\sigma(\mathcal I)$ as in Definition \ref{defn:subtree_sigma}, which consists of the colored vertices, the siblings of the colored vertices, the branches connecting them to the root, and the siblings of the vertices on the branches.}
\label{fig:tree_subtree}
\end{figure}
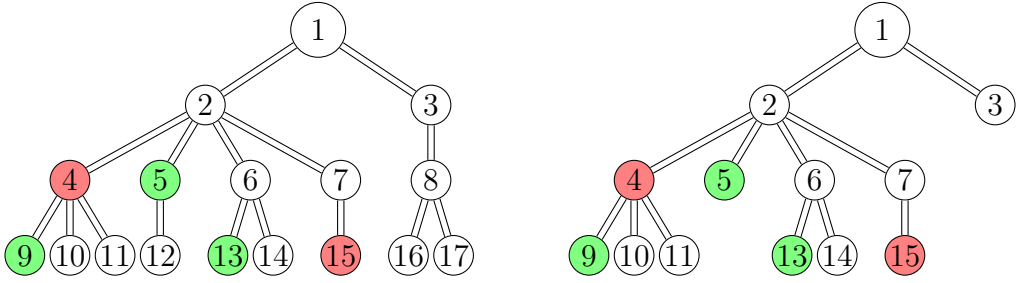

\begin{defn}
\label{defn:sim_classes}
Let $\mathcal{I}_1$ and $\mathcal{I}_2$ be two tree isomorphism classes satisfying $\Delta(\mathcal{I}_1) = \Delta(\mathcal{I}_2) = 0$, with $k_1$ and $k_2$ vertices, respectively.  
Let $\sigma \in P_{k_1}$ and $\tau \in P_{k_2}$ be set partitions.  
Suppose further that $\mathcal{I}_1$ and $\mathcal{I}_2$ arise from Wick pairings of unlabeled trees $T_1$ and $T_2$.

We say that the pairs $(\mathcal{I}_1, \sigma)$ and $(\mathcal{I}_2, \tau)$ are \textbf{similar}, written
\[
(\mathcal{I}_1, \sigma) \sim (\mathcal{I}_2, \tau),
\]
if the subtrees $\sigma(\mathcal{I}_1)$ and $\tau(\mathcal{I}_2)$ are obtained in an identical way from their respective Wick pairings of $T_1$ and $T_2$, and if the blocks of non-singleton vertices coincide between $\sigma(\mathcal{I}_1)$ and $\tau(\mathcal{I}_2)$.  

Equivalently, $(\mathcal{I}_1, \sigma) \sim (\mathcal{I}_2, \tau)$ whenever the same subset of edges in $T_1$ and $T_2$ are paired under the Wick procedure, and the partitions $\sigma$ and $\tau$ have identical configurations of non-singleton blocks within the resulting subtrees.
\end{defn}

The Definition \ref{defn:sim_classes} is motivated by the following Proposition, which allows for the control of the sum over similarity classes. In summary, the proposition states that we can fix the subtree $\sigma(\mathcal I)$ and consider all possible ``extensions" to a full isomorphism class $\mathcal J$. For each leaf vertex $v$ in $\sigma(\mathcal I)$, if we sum over all possible extensions, we will get exactly the term $\tau_{s_v}^2$, where $s_v$ denotes the generation of $v$. 

\begin{prop}
\label{prop:similarity_sum}
Let $\mathcal{I}$ be a tree isomorphism class with $\Delta(\mathcal{I}) = 0$ and $|V(\mathcal{I})| = k$, and let $\sigma \in P_k$ be a partition.  
Denote by $\mathcal{S}(\mathcal{I}, \sigma)$ the similarity class of $(\mathcal{I}, \sigma)$, consisting of all pairs $(\mathcal J, \tau)$ such that $(\mathcal J, \tau) \sim (\mathcal{I}, \sigma)$, where each $\mathcal J$ belongs to $\mathcal{I}_{i,s,2m,0}$, and $\tau$ is a partition of its vertex set.

Let $T_{\mathcal{I}}$ denote the underlying monomial tree whose Wick pairing produces the class $\mathcal{I}$, and let $T_{\sigma(\mathcal{I})}$ denote the corresponding subtree of $T_{\mathcal{I}}$ that was paired to form the subtree associated with $\sigma(\mathcal{I})$.

Then
\[
\sum_{(\mathcal J, \tau) \in \mathcal{S}(\mathcal{I}, \sigma)}
\left(
  \prod_{v \in V(T_{\mathcal J})} c_v
  \cdot
  \prod_{\ell \in L(T_{\mathcal J})} \tau_0
\right)
=
\left(
  \prod_{v \in V(T_{\sigma(\mathcal{I})})} c_v
\right)
\cdot
\left(
  \prod_{v \in L(\sigma(\mathcal{I}))} \tau_{s_v}^2
\right),
\]
where $s_v$ denotes the generation (depth) of the vertex $v$ in the monomial tree, and $L(\sigma(\mathcal I))$ denotes the set of leaf vertices of $\sigma(\mathcal I)$.
\end{prop}

\begin{proof}
For each element $(\mathcal J, \tau) \in \mathcal{S}(\mathcal{I}, \sigma)$ the subtree $\tau(\mathcal J)$ corresponds to the same pairing pattern of the underlying monomial tree as $\sigma(\mathcal I)$ and, hence, to the same subtree $T_{\sigma(\mathcal{I})}=T_{\tau(\mathcal J)}$.  
This means the contributions from the coefficients $c_v$ are the same for all such subtrees.  
Factoring out this common component yields
\[
\prod_{v \in V(T_{\sigma(\mathcal{I})})} c_v
\]

The remaining contribution comes from summing over all possible extensions of the leaves of the subtree $\sigma(\mathcal I)$ within the similarity class.  
For each leaf $v \in L(\sigma(\mathcal{I}))$ and element $(\mathcal J, \tau) \in \mathcal{S}(\mathcal{I}, \sigma)$, we can consider the subtree $\mathcal J_v$, which is the subtree of $\mathcal J$ rooted at $v$. Note that this makes sense since $\tau(\mathcal J)$ and $\sigma(\mathcal I)$ are the same.  Suppose $v$ belongs to generation $s_v$. Then we observe that $\mathcal J_v$ itself is an isomorphism class with $\Delta(\mathcal J_v)=0$ of trees of depth $s_v$. The subtree $\sigma(\mathcal I)$ is fixed, but there are no constraints on the independent Wick-paired subtrees rooted at its boundary. By the Lemma \ref{lem:delta_0_val_2} on summation over isomorphism classes, the total contribution of all such extensions $\mathcal J_v$ equals $\tau_{s_v}^2$.   Since these extensions are independent across leaves, the factors multiply:
\[
\prod_{v \in L(\sigma(\mathcal{I}))} \tau_{s_v}^2.
\]

Combining the contributions of the common bad subtree and the leaf extensions gives
\[
\prod_{v \in V(T_{\sigma(\mathcal{I})})} c_v
\cdot
\prod_{v \in L(\sigma(\mathcal{I}))} \tau_{s_v}^2,
\]
which is the desired identity. Note that if a vertex $v\in L(\sigma(\mathcal I))$ is also a leaf vertex in $T_\mathcal I$, we obtain the factor $\tau_0^2$.
\end{proof}

\begin{lem}
    \label{lem:number_similarity_classes}
    Recall the setting of Theorem~\ref{thm:State_evolution_approximation_polynomial},
    where $t = (\log N)/C_D$ with \[C_D = K_1 D \log(\max\{D,M\})\]
    Let $\mathcal I_{i,s,m,0}$ be the set of isomorphism classes with excess~$0$
    appearing in the expansion of $(x_i^s)^m$.
    For each integer $r \ge 1$, let $\mathcal Q_r$ denote the set of similarity classes
    \[
        \mathcal Q_r
        := \bigl\{\mathcal S(\mathcal I,\sigma)
            : \mathcal I \in \mathcal I_{i,s,m,0},\;
              \sigma \in P_{|V(\mathcal I)|},\;
              |V(\mathcal I)| - |\sigma| = r
        \bigr\}.
    \]
    Then
    \[
        |\mathcal Q_r|
        \le
        N^{4r/K_1}\bigr.
    \]
\end{lem}

\begin{proof}
    A similarity class $\mathcal S(\mathcal I,\sigma)$ is uniquely determined by
    the subtree $\sigma(\mathcal I)$ constructed in Definition~\ref{defn:subtree_sigma}.
    Thus, it suffices to bound the total number of distinct subtrees $\sigma(\mathcal I)$
    arising from all possible pairs $(\mathcal I,\sigma)$ with $|V(\mathcal I)| - |\sigma| = r$.

    \smallskip
    Step 1: Choosing the non-singleton vertices.
    A partition $\sigma$ with $|V(\mathcal I)| - |\sigma| = r$
    has exactly $r$ ``collisions,''
    meaning $r$ vertices lie in blocks of size at least~$2$.
    These $r$ vertices form precisely the set of initial vertices
    added in step~(i) of Definition~\ref{defn:subtree_sigma}.

    For each such vertex, there is a unique simple path of doubled edges from it to the root, which is added in step (ii) of Definition~\ref{defn:subtree_sigma}.
    Any class $\mathcal I$ has at most $C_0 D^t(\log \log N) \le N^{1/K_1}$
    non-root vertices. The $\log \log N$ factor comes from the choice of the degree of the root vertex $m\leq \log \log N$. For each of the $r$ collision vertices, we attach a subpath of a simple path to the root. The length of this subpath is at most $t$, and there are at most $N^{1/K_1}$ possible locations to glue the subpath.
    Hence, the total number of choices for the $r$ initial vertices and the paths connecting them to the root is at most
    \(
        (t N^{1/K_1})^{r}
    \)

    \smallskip
    Step 2: Local degrees and Wick pairings.
    The remaining structure of the subtree $\sigma(\mathcal I)$ is related to the siblings of the vertices added in Step (1), (see step (ii) of Definition~\ref{defn:subtree_sigma}).
    This structure is determined by the local monomial degrees and Wick pairing information in the
    underlying unlabeled tree. In Step (1), we added at most $rt$ vertices, and for each vertex we have $O(D^D)$ ways to choose the underlying unlabeled tree structure of that vertex and its siblings.
    These contribute at most $(2D)^{O(Drt)}(\log \log N)^{\log \log N} = N^{o(r)}$ possibilities,
    which are no larger than $(N^{1/K_1})^{r}$.

    \smallskip
    \noindent
    Combining the estimates above, the total number of distinct subtrees
    $\sigma(\mathcal I)$ with defect $r$ is at most
    \(
        N^{4r/K_1}.
    \)
    Since each such subtree corresponds to exactly one similarity class,
    the same bound holds for $|\mathcal Q_r|$.
    This proves the lemma.
\end{proof}

\begin{prop}\label{prop:sum_delta_0_accurate}
Recall the setting of Theorem~\ref{thm:State_evolution_approximation_polynomial}, 
where $ t = (\log N)/C_D $ with $C_D = K_1 D \log(\max\{D, M\})$, and
$c_{s,d}, \tau_s < M$ for all iterations. Then
\[
\sum_{\mathcal I \in \mathcal I_{i,s,m,0}} 
\mathbb E\,\mathrm{Val}(\mathcal I)
= \mathbb E\left[(y_i^s)^{m}\right]+ O(N^{-1 + 10/K_1}).
\]
Where $y_i^s\sim \mathcal N(0, \tau_s^2)$.
\end{prop}

\begin{proof}
 Note that when $m$ is odd, both sides of the equation hold trivially since there are no isomorphism classes with $\Delta(\mathcal I)=0$ and $\mathbb E[(y_i^s)^m]=0$. So we consider $2m$, which is even. Let $\mathcal I \in \mathcal I_{i,s,2m,0}$ be an isomorphism class with 
    $k = |V(\mathcal I)|$.  
    By Proposition~\ref{prop:factorial_formula},
    \[
    N^k - (N)_k
    = \sum_{\sigma \in P_k \setminus \{\hat\sigma\}}
    (-1)^{k - |\sigma| + 1}
    \left(\prod_{B \in \sigma} (|B|-1)!\right) N^{|\sigma|},
    \]
    where $\hat\sigma$ denotes the trivial partition ($k$ blocks of size $1$).  
    Here $(N)_k$ counts injective labelings of $\mathcal I$, 
    and thus $N^k - (N)_k$ counts non-injective labelings, 
    decomposed by partitions $\sigma$ in which each block consists of vertices sharing a label.

       Recall that by Lemma \ref{lem:delta_0_val_2} and Lemma \ref{lem:non_injective} we must bound the RHS of \eqref{eq:difference_injective} to prove the Proposition. We may rewrite the RHS of \eqref{eq:difference_injective} in the following way,
    \begin{align}
    \label{eq:sum_partitions}
    \sum_{\mathcal I\in \mathcal I_{i, s, 2m,0}}\sum_{\sigma\in P_{|V(\mathcal I)|}\backslash \{\hat \sigma\}}(-1)^{|V(\mathcal I)|-|\sigma|+1}\left(\prod_{B\in \sigma}(|B|-1)!\right)\cdot N^{|\sigma|}\cdot  N^{-|V(\mathcal I)|} \prod_{v \in V(T_\mathcal I)}  c_v \cdot \prod_{\ell \in L(T_\mathcal I)}\tau_0
    \end{align}
    Consider the sum \eqref{eq:sum_partitions} restricted to a single similarity class in the sense of Definition \ref{defn:sim_classes}. That is considered the sum,
    \[
    \sum_{(\mathcal J, \tau)\in (\mathcal I, \sigma) } (-1)^{|V(\mathcal J)|-|\tau|+1}\left(\prod_{B\in \tau}(|B|-1)!\right)\cdot N^{|\tau|-|V(\mathcal J)|} \prod_{v \in V(T_\mathcal J)}  c_v \cdot \prod_{\ell \in L(T_\mathcal J)}\tau_0
    \]
    Within a class, the factors
    \[
    (-1)^{|V(\mathcal J)| - |\tau| + 1},\qquad
    \prod_{B\in\tau} (|B|-1)!,\qquad
    N^{|\tau| - |V(\mathcal J)|}
    \]
are constant, depending only on the partition $\tau$.  This is because, the three factors only depends on the subtree $\tau(\mathcal J)$ which is fixed over the similarity class.
    By Proposition \ref{prop:similarity_sum} the sum over a single similarity class factors to:
    \begin{align}
    \label{eq:sim_class_sum}
    (-1)^{|V(\mathcal I)|-|\sigma|+1} \left(\prod_{B\in \sigma}(|B|-1)!\right) N^{|\sigma|-|V(\mathcal I)|}
  \prod_{v \in V(T_{\sigma(\mathcal{I})})} c_v
  \prod_{v \in L(\sigma(\mathcal{I}))} \tau_{s_v}^2
    \end{align}
    The number of generations in the tree is at most $t\leq \log N/C_D$, $C_D=K_1\cdot D\log(\max\{D, M\})$. The quantity $|V(\mathcal I)| - |\sigma|$ counts how many vertex identifications occur. In the construction of $\sigma(\mathcal I)$ 
    (Definition~\ref{defn:subtree_sigma}):
    
    \begin{itemize}
    \item Step $(i)$ adds at most $|V(\mathcal I)| - |\sigma|$ vertices.
    \item Step $(ii)$ adds at most $tD$ new vertices for each added vertex in step $(i)$.
    \end{itemize}
    Hence
    \[
    |V(T_{\sigma(\mathcal I)})| \le tD (|V(\mathcal I)| - |\sigma|).
    \]
    Since $c_v, \tau_{s_v}<M$ we may bound the factor,
    \[
    \prod_{v \in V(T_{\sigma(\mathcal I)})}  c_v \cdot \prod_{\ell \in L(T_{\sigma(\mathcal I)})}\tau_0\leq M^{4tD(|V(\mathcal I)|-|\sigma|)}\leq N^{4(|V(\mathcal I)|-|\sigma|)/K_1}
    \]
    Now we bound the factor,
    \[
    \left(\prod_{B\in \sigma}(|B|-1)!\right)\leq |V(\mathcal I)|^{|V(\mathcal I)|-|\sigma|}\leq N^{(|V(\mathcal I)|-|\sigma|)/K_1}
    \]
    Similarly, we bound the factor, \[\prod_{v \in L(\sigma(\mathcal{I}))} \tau_{s_v}^2\leq M^{tD(|V(\mathcal I)|-|\sigma|)}\leq N^{(|V(\mathcal I)|-|\sigma|)/K_1}\]
    
    Combining these bounds in \eqref{eq:sim_class_sum}, 
    each similarity class with parameter pair $(\mathcal I,\sigma)$ contributes at most
    \[
      \Bigl|\mathrm{contrib}(\mathcal I,\sigma)\Bigr|
      \;\lesssim\;
      N^{(|\sigma| - |V(\mathcal I)|)(1 - 6/K_1)}.
    \]
    Writing $r = |V(\mathcal I)| - |\sigma| \ge 1$, this is
    \[
      \Bigl|\mathrm{contrib}(\mathcal I,\sigma)\Bigr|
      \;\lesssim\;
      N^{-r(1 - 6/K_1)}.
    \]

    Now group the similarity classes according to this parameter $r$.  
    For each $r \ge 1$, let $\mathcal Q_r$ denote the set of similarity classes 
    $\mathcal S(\mathcal I,\sigma)$ with $|V(\mathcal I)| - |\sigma| = r$.  
    By Lemma~\ref{lem:number_similarity_classes}, we have
    \[
      |\mathcal Q_r| \le N^{4r/K_1}.
    \]
    Therefore, the total contribution of all similarity classes with parameter $r$ is bounded by
    \[
      \sum_{\mathcal S(\mathcal I,\sigma)\in \mathcal Q_r}
      \Bigl|\mathrm{contrib}(\mathcal I,\sigma)\Bigr|
      \lesssim
      |\mathcal Q_r|\cdot N^{-r(1 - 6/K_1)}
      \le
      N^{4r/K_1} N^{-r(1 - 6/K_1)}
      =
      N^{-r(1 - 10/K_1)}.
    \]

    Summing over $r\ge 1$ gives
    \[
      \sum_{r\ge 1}
      \sum_{\mathcal S(\mathcal I,\sigma)\in \mathcal Q_r}
      \Bigl|\mathrm{contrib}(\mathcal I,\sigma)\Bigr|
      \lesssim
       N^{-1 + 10/K_1},
    \]
    where we used that the worst case occurs at $r=1$ and absorbed the finite sum
    over $r$ into the implicit constant.
    Therefore, the right-hand side of \eqref{eq:difference_injective}, i.e.\ the contribution
    from all non-injective labelings, is
    \( O\bigl(N^{-1 + 10/K_1}\bigr).\)
    By Lemma~\ref{lem:delta_0_val_2} and Lemma~\ref{lem:non_injective}, this shows that
    the total contribution of non-injective labelings to
    \[
      \sum_{\mathcal I\in \mathcal I_{i,s,2m,0}}
      \mathbb E\,\mathrm{Val}(\mathcal I)
    \]
    is $O(N^{-1 + 10/K_1})$. On the other hand, by Lemma~\ref{lem:delta_0_val_2} the contribution of all labelings from the classes
    $\mathcal I_{i,s,2m,0}$ is exactly \(\mathbb E\bigl[(y_i^s)^{2m}\bigr]\) where $y_i^s\sim\mathcal N(0,\tau_s^2)$.
    Combining these, we obtain
    \[
      \sum_{\mathcal I\in \mathcal I_{i,s,2m,0}}
      \mathbb E\,\mathrm{Val}(\mathcal I)
      \;=\;
      \mathbb E[(y_i^s)^{2m}]
       + O(N^{-1 + 10/K_1})\bigr).
    \]
    Where on the LHS we sum over injective labelings of the isomorphism classes.
\end{proof}

We use the notation \( T(x_i^s) \) to denote the element of the moment algebra \( \mathcal{A} \) corresponding to the iterate \( x_i^s \). Explicitly,  
\[
T(x_i^s) \;=\; \sum_{T \in {T}_{i,s}}
    T \cdot \prod_{v \in V(T)} c_v \cdot \prod_{\ell \in L(T)} \tau_0,
\]
where \( {T}_{i,s} \) is the collection of unlabeled trees associated with the iterate \( x_i^s \). Furthermore, we write $T(x_i^{s_1})\star T(x_i^{s_2})$ to denote the expansion of $x_i^{s_1}\cdot x_i^{s_2}$. That is,
\[
T(x_i^s) \;=\; \sum_{T_1 \in {T}_{i,s_1}}\sum_{T_2 \in {T}_{i,s_2}}
    T_1\star T_2 \cdot \prod_{v \in V(T_1), V(T_2)} c_v \cdot \prod_{\ell \in L(T_1), L(T_2)} \tau_0
\]

The analog of Proposition \ref{prop:sum_delta_0_accurate} also holds for the expansion of $x_i^{s_1}\cdot x_i^{s_2}$. That is,

\begin{prop}
\label{prop:delta_zero_covariance}
    We have the following
    \[
    \sum_{\mathcal I\in \mathcal I_{i,s_1,s_2,0}}\mathbb E\mathrm{Val}(\mathcal I)=\mathrm{Wick}(T(x_i^{s_1})\star T(x_i^{s_2}))+O(N^{-1+10/K_1}))
    \]
    Where $\mathcal I_{i,s_1,s_2,0}$ denotes isomorphism classes of trees with excess $\Delta=0$ produced from unlabeled trees in the expansion $T(x_i^{s_1})\star T(x_i^{s_2})$. 
\end{prop}

\begin{proof}
    The proof is essentially the same as Proposition \ref{prop:sum_delta_0_accurate}. We recall that by Lemma \ref{lem:delta_0_perfect_pairing} isomorphism classes of $\mathcal I_{i,s_1,s_2,0}$ are perfectly Wick paired. Then we can repeat the arguments of Proposition \ref{prop:sum_delta_0_accurate} to obtain the result.
\end{proof}

\subsection{Bounding the contribution from \texorpdfstring{$\Delta>0$}{Delta>0} labeled trees}

Proposition \ref{prop:sum_delta_0_accurate} shows that the contribution from the $\Delta=0$ trees asymptotically matches the state evolution formulas. It remains to show that the contribution from the $\Delta>0$ trees is small. Recall that $\mathcal T_{i, t, m}$ denotes the set of labeled monomial trees of depth $t$ rooted at vertex $i$, such that the root has $m$ children.

\begin{defn}
    Given a labeled tree in $\mathcal T_{i, t, m}$ or a tree isomorphism class of labeled trees in $\mathcal T_{i, t, m}$. We define a vertex in that tree to be \textbf{bad} if any of the following holds:
\begin{itemize}
    \item It has label \( i \) with multiplicity \( N_i \geq 3 \).
    \item It is incident to a \( \mathbf{T_3} \) edge or shares a label with a vertex incident to a $\bf{T_3}$ edge (Definition \ref{def:edge_types}).
    \item It shares the same label as the root (but is not the root).
    \item It shares a label with a vertex that satisfies any of the above.
\end{itemize}
Let all other vertices be \textbf{good}. 
\end{defn}

\begin{remark}
    The good vertices are the vertices whose label $i$ occurs twice $(N_i=2)$, and the edges incident to those two vertices with label $i$ are all $\bf{T_1}$ and $\bf{T_2}$, meaning they are ``Wick Paired." 
\end{remark}

\begin{defn} \label{defn:bad_subtree}
Let $\mathcal T$ be a labeled tree (or labeled tree isomorphism class).  
For each bad vertex $v \in \mathcal T$, define its \textbf{branch} $b(v)$ to be the unique simple path in $\mathcal T$ from $v$ to the root, including both $v$ and the root.  

The \textbf{bad subtree} $\mathcal B$ of $\mathcal T$ is the subgraph consisting of:
\begin{itemize}
    \item[$(i)$] all vertices and edges lying in each branch $b(v)$, for every bad vertex $v$.
    \item[$(ii)$] all vertices with the same label as any vertex on some branch $b(v)$ from step $(i)$. We call the vertices from steps $(i)$ and $(ii)$ the \textbf{branch} vertices of the bad subtree. In addition, add the edge between two branch vertices if they are incident in $\mathcal T$
    \item[$(iii)$] every additional edge of $\mathcal T$ that is incident to a branch vertex from step $(i)$ or $(ii)$, together with both of its endpoints.  The vertices added in this step are called \textbf{boundary vertices}, which together form the \textbf{boundary} $\partial \mathcal B$ of the bad subtree.
\end{itemize}
We remark that the definition of the bad subtree is the fixed (up to permutation of the labels) for a labeled tree isomorphism class since the equivalence classes of vertex labels are determined by the isomorphism class. Note that in this definition $\partial \mathcal B\subset \mathcal B$. 
\end{defn}

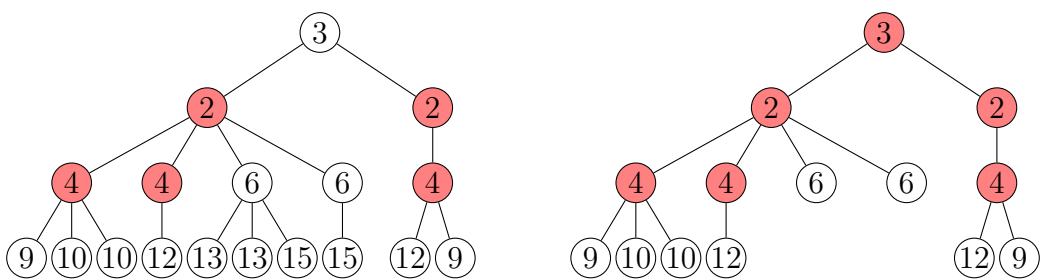
\begin{figure}[h]
\centering
\begin{tikzpicture}[level distance=1cm,
  level 1/.style={sibling distance=3cm},
  level 2/.style={sibling distance=1.2cm},
  level 3/.style={sibling distance=0.6cm}
  ]

\node[circle,draw,fill=white, minimum size=15pt, inner sep=0pt] (r1) at (0,0) {3}
  child {node[circle,draw,fill=red!50, minimum size=15pt, inner sep=0pt] (a1) {2}
    child {node[circle,draw,fill=red!50, minimum size=15pt, inner sep=0pt] (a11) {4}
        child {node[circle,draw,fill=white, minimum size=15pt, inner sep=0pt] (a111) {9}}
        child {node[circle,draw,fill=white, minimum size=15pt, inner sep=0pt] (a112) {10}}
        child {node[circle,draw,fill=white, minimum size=15pt, inner sep=0pt] (a113) {10}}
    }
    child {node[circle,draw,fill=red!50, minimum size=15pt, inner sep=0pt] (a12) {4}
        child {node[circle,draw,fill=white, minimum size=15pt, inner sep=0pt] (a121) {12}}}
    child {node[circle,draw,fill=white, minimum size=15pt, inner sep=0pt] (a13) {6}
        child {node[circle,draw,fill=white, minimum size=15pt, inner sep=0pt] (a131) {13}}
        child {node[circle,draw,fill=white, minimum size=15pt, inner sep=0pt] (a132) {13}}
        child {node[circle,draw,fill=white, minimum size=15pt, inner sep=0pt] (a132) {15}}}
    child {node[circle,draw,fill=white, minimum size=15pt, inner sep=0pt] (a14) {6}
        child {node[circle,draw,fill=white, minimum size=15pt, inner sep=0pt] (a141) {15}}}
  }
  child {node[circle,draw,fill=red!50, minimum size=15pt, inner sep=0pt] (a2) {2}
    child {node[circle,draw,fill=red!50, minimum size=15pt, inner sep=0pt] (a21) {4}
         child {node[circle,draw,fill=white, minimum size=15pt, inner sep=0pt] (a211) {12}}
        child {node[circle,draw,fill=white, minimum size=15pt, inner sep=0pt] (a212) {9}}
    }
  };

\node[circle,draw,fill=red!50, minimum size=15pt, inner sep=0pt] (r1) at (7.5,0) {3}
  child {node[circle,draw,fill=red!50, minimum size=15pt, inner sep=0pt] (a1) {2}
    child {node[circle,draw,fill=red!50, minimum size=15pt, inner sep=0pt] (a11) {4}
        child {node[circle,draw,fill=white, minimum size=15pt, inner sep=0pt] (a111) {9}}
        child {node[circle,draw,fill=white, minimum size=15pt, inner sep=0pt] (a112) {10}}
        child {node[circle,draw,fill=white, minimum size=15pt, inner sep=0pt] (a113) {10}}
    }
    child {node[circle,draw,fill=red!50, minimum size=15pt, inner sep=0pt] (a12) {4}
        child {node[circle,draw,fill=white, minimum size=15pt, inner sep=0pt] (a121) {12}}}
    child {node[circle,draw,fill=white, minimum size=15pt, inner sep=0pt] (a13) {6}}
    child {node[circle,draw,fill=white, minimum size=15pt, inner sep=0pt] (a14) {6}}
  }
  child {node[circle,draw,fill=red!50, minimum size=15pt, inner sep=0pt] (a2) {2}
    child {node[circle,draw,fill=red!50, minimum size=15pt, inner sep=0pt] (a21) {4}
         child {node[circle,draw,fill=white, minimum size=15pt, inner sep=0pt] (a211) {12}}
        child {node[circle,draw,fill=white, minimum size=15pt, inner sep=0pt] (a212) {9}}
    }
  };
  
\end{tikzpicture}
\caption{(Left) A labeled tree $\mathcal T$ with bad vertices colored in red ($N_4=3$ the red label of $i=2$ comes due to incidence to a $\bf{T_3}$ edge). (Right) The bad subtree 
$\mathcal B$ of $\mathcal T$. The red vertices denote the branch vertices of $\mathcal B$ and the white vertices are the boundary vertices of $\mathcal B$}
\label{fig:bad_subtree}
\end{figure}

\begin{prop}
\label{prop:bad_subtree_properties}
    Let $\mathcal T$ be a labeled tree and let $\mathcal B$ be its bad subtree, as in Definition~\ref{defn:bad_subtree}. Then:
    \begin{itemize}
        \item[$(1)$] The branch vertices of $\mathcal B$ form a connected subtree containing the root. In particular, the unique simple path from branch vertex to the root must consist of branch vertices.
        \item[$(2)$] The boundary vertices of $\mathcal B$ are all good and are paired
        \item[$(3)$] The subtrees rooted at each boundary vertex $v$ (glued with the pair vertex $w$ from step $(2)$) contain only good vertices not in $\mathcal B$ (except for $v,w$) and are row-wise ``Wick Paired" in the sense of Definition \ref{defn:tree_wick_pairings}.
    \end{itemize}
\end{prop}

\begin{proof}
    $(1)$. Suppose for contradiction that the branch vertices of $\mathcal B$ have another component $\mathcal C$ that does not contain the root. Then $\mathcal C$ contains no bad vertices as that would lead to a path of branch vertices from the bad vertex to the root. Hence, $\mathcal C$ must contain a branch vertex $u$ added by step $(ii)$ of Definition~\ref{defn:bad_subtree}. Let $\tilde{u}$ be the branch vertex (in the root component) with the same label as $u$, which was added in step $(i)$ i.e., it lies on a path from a bad vertex to the root.  

    Let
\[
\tilde{u} = \tilde v_0, \tilde v_1, \dots, \tilde v_k = \text{root}
\]
denote the unique path from $\tilde{u}$ to the root. Observe, $\tilde u$ is a good vertex. There must be another edge between vertices with the same labels as $(\tilde u, \tilde v_1)$ (since each edge appears twice). By the non-backtracking condition, there must be a vertex incident to $u$ with the same label as $\tilde v_1$, which we call $v_1$. Then $v_1\in \mathcal C$ since $\tilde v_1$ is a branch vertex. Repeat the argument for $v_1$, to obtain a vertex $v_2\in \mathcal C$ with the same label as $\tilde v_2$. Continuing to $v_k$ we conclude that $\mathcal C$ contains a vertex with the same label as the root (but is not the root), which is a bad vertex, contradicting that $\mathcal C$ contains no bad vertices.  
Hence, the branch vertices of $\mathcal B$ are a connected subtree that contains the root. Since there is only one path from $u$ to the root this path consists of only branch vertices.

(2). Consider any boundary vertex $v\in \partial \mathcal B$, which, by definition, is a good vertex; otherwise, it would be a branch vertex. First, observe that the branch vertex $u$ incident to $v$ must be the parent of $v$. If $u$ was the child of $v$, then the path from $u$ to the root contains $v$, and by $(1)$, it implies that $v$ must be a branch vertex. There exists another vertex $\tilde v$ with the same label as $v$. $\tilde v$ must be incident to a branch vertex $\tilde u$ since the edges with labels of $(u, v)$ must appear twice, and $u$ is the parent of $v$, so it cannot pair with a child of $v$ by the non-backtracking assumption. Hence, we know there is another boundary vertex $\tilde v$ paired to $v$.

(3). As mentioned in $(2)$, the subtree rooted at each boundary vertex $v\in \partial \mathcal B$ contains no branch vertices, and hence no vertices of $\mathcal B$. All vertices in the subtree are good. Otherwise, there would be a path of branch vertices from a bad vertex to the root passing through $v$, making $v$ a branch vertex. Let $\tilde v\in \partial \mathcal B$ denote the pair vertex of $v$ from $(2)$. Then, the labels of the children of $\tilde v$ and $v$ must all be paired since the labels of $\tilde v$ and $v$ do not appear elsewhere. Continuing this argument, we see that the subtree obtained by gluing the subtrees rooted at $v$ and $\tilde v$ is perfectly Wick paired since it only contains good vertices.
\end{proof}

\begin{prop} \label{prop:num_bad_subtrees}
Let $\mathcal I_{i,s,m,\Delta}$ denote the set of isomorphism classes of
labeled monomial trees with excess $\Delta>0$.
Then the number of distinct bad subtrees $\mathcal B(\mathcal I)$ that can
occur among all $\mathcal I \in \mathcal I_{i,s,m,\Delta}$ is at most
\(
    N^{C_4\Delta/K_1},
\)
for some absolute constant $C_4>0$.
\end{prop}

\begin{proof}
By Lemma~\ref{lem:Ni_3}, at most $6\Delta$ vertices satisfy $N_i\ge 3$,
and at most $2\Delta$ vertices share the root label.
By Lemma~\ref{lem:t4_bound}, there are at most $2\Delta$ ${\bf T_3}$--edges,
each contributing two bad vertices.
Propagating badness through repeated labels
(Definition~\ref{defn:bad_subtree})
increases this by only a constant factor.
Hence
\(|\mathrm{Bad}(\mathcal I)| \le 26\Delta .
\) where $|\mathrm{Bad}(\mathcal I)| $ is the number of bad vertices in $\mathcal I$.

\smallskip
Step 1: Adding paths to the root. Each bad vertex contributes the entire path to the root.
Since each path has length at most $t$, this adds at most
\(
    C_0 \Delta t
\)
additional vertices.
The number of good branch vertices introduced in this way is at most twice this
(there are at most two appearances of each label), so the total number of
branch vertices is bounded by
\(
    C_1 \Delta t .
\)

\smallskip
Step 2: Attaching boundary vertices.
Every non-root branch vertex has a degree of at most $D$ and the root has degree at most $\log \log N$.
Thus, attaching its neighboring boundary vertices contributes at most \(
    C_1 \Delta t D\log \log N \)
additional vertices.
Absorbing constants, we obtain
\(|\mathcal B(\mathcal I)| \le C_1 t \Delta D \log \log N 
\) where $|\mathcal B(\mathcal I)|$ is the number of vertices in the bad subtree.
Set
\( L:=C_1 t \Delta D \log \log N .
\)

\smallskip
Step 3: Counting possible unlabeled shapes.
A rooted tree with at most $L$ vertices and out-degree at most $D$
has at most
\(
    ( C_2 D )^{L}
\)
possible unlabeled shapes (each vertex chooses a degree in $\{0,\dots,D\}$). Once the unlabeled shape is fixed it may be embedded in the underlying unlabeled tree in also at most \(
    ( C_2 D )^{L}
\) ways. This follows since for each vertex we have already chosen its degree and so there are at most $D$ options for the underlying unlabeled tree vertex it corresponds to. We combine the two factors, possibly changing the constant $C_1$.

\smallskip
Step 4: Identifying labels of vertices in $\mathcal B(\mathcal I)$.

(a) Choosing the bad vertices.
There are at most $26\Delta$ bad vertices contained in a total of $L$ vertices,
so the number of ways to choose them and assign their labels is at most
\(
    L^{26\Delta}\, (26\Delta)^{26\Delta}.
\)

\smallskip
(b) Identifying vertices sharing a bad label.
Fix a bad label $i$ that appears on vertices $v_1,\dots,v_k$.
Each $v_j$ has $\ell_j$ children and one parent, so the total number of
incident edges to be grouped is
\(
    \ell_1+\cdots+\ell_k+k.
\)
The number of partitions of these edges is bounded by
\[
    (\ell_1+\cdots+\ell_k+k)!\, 2^{\ell_1+\cdots+\ell_k+k}
    \;\le\;
    (\ell_1+\cdots+\ell_k+k)^{\ell_1+\cdots+\ell_k+k}.
\]
Since $k\le\Delta$ and $\ell_j\le D$, $m\leq \log \log N$, we obtain the bound
\(
    (26\Delta(D+1)+\log \log N)^{26\Delta(D+1)+\log \log N} 
\). If $\Delta D\leq \log \log N$ we can upper bound,
\[
    (26\Delta(D+1)+\log \log N)^{26\Delta(D+1)+\log \log N} \leq (C_3\log \log N)^{C_3\log \log N}\leq N^{1/K_1}
\]
Otherwise, the sum is bounded by
\[
(52\Delta(D+1))^{52\Delta(D+1)}
\]

\smallskip
(c) Pairings along good branches.
Along each branch from a bad vertex, every good vertex $u$
(with $N_i=2$ and incident only to ${\bf T_1}$ and ${\bf T_2}$ edges)
admits at most $(2D)^{2D}$ ways to pair all its neighbors with the other vertex
sharing its label.
Each branch has length at most $t$, and there are at most $52\Delta t D$
such vertices overall, giving a total contribution of
\(
    (2D)^{52\Delta t D}.
\)

Multiplying all contributions yields
\[
\begin{aligned}
&\quad (C_2 D)^L
\; L^{26\Delta} (26\Delta)^{26\Delta}
\; (52\Delta(D+1))^{52\Delta(D+1)}
\cdot N^{1/K_1}\cdot (2D)^{52\Delta t D}\cdot
\end{aligned}
\]
Since
\(
    \Delta \le |V(T)| \le D^t \le N^{1/(DK_1)},
\)
all factors are bounded by
\(
    N^{C_4\Delta/K_1}
\)
for an absolute constant $C_4>0$.
This proves the claim.
\end{proof}

\begin{remark}
\label{rem:growth_rate}
    We remark that in the proof of Proposition \ref{prop:num_bad_subtrees}, a factor of the form $(\Delta D)^{\Delta D}$ appears. This is why it is necessary to take $t\lesssim \log N/(D\log D)$ as opposed to the expected $t\lesssim \log N/(\log D)$. It appears that such a factor is unavoidable in this proof. For instance, there may be trees with $\Delta$ vertices with the same label, and then there will be roughly $(\Delta D-1)!!$ ways to pair the edges incident to those vertices.
\end{remark}

Recall that for a labeled tree \( \mathcal{T} \), its value is defined as
\[
\operatorname{Val}(\mathcal{T})
    = \prod_{v \in V(\mathcal{T})} c_v 
      \cdot \prod_{(u, v) \in E(\mathcal{T})} A_{\phi(u)\phi(v)} 
      \cdot \prod_{\ell \in L(\mathcal{T})} x^0_\ell.
\]
For a bad subtree \( \mathcal{B} \subseteq \mathcal{T} \), we define analogously.
\[
\operatorname{Val}_1(\mathcal{B})
    = \prod_{v \in V(\mathcal{B})\setminus \partial \mathcal B} c_v 
      \cdot \prod_{(u, v) \in E(\mathcal{B})} A_{\phi(u)\phi(v)}\cdot \prod_{\ell \in L(\mathcal{B})\setminus \partial \mathcal B} x^0_\ell
\]
Where $L(\mathcal B)$ denotes the vertices of $\mathcal B$ that are leaf vertices in $\mathcal T$. So, $L(\mathcal B)\setminus \partial \mathcal B$ are the branch vertices that are also leaf vertices of $\mathcal T$. Further, for any subtree $\mathcal S\subset \mathcal T$, define,
\[
\operatorname{Val}_2(\mathcal{S})
    = \prod_{v \in V(\mathcal{S})} c_v 
      \cdot \prod_{(u, v) \in E(\mathcal{S})} A_{\phi(u)\phi(v)} 
      \cdot \prod_{\ell \in L(\mathcal{S})} x^0_\ell.
\]
Furthermore, we define $\Delta(\mathcal B)=\frac{1}{2}|E(\mathcal B)|-|V(\mathcal B)|+1$, where $|V(\mathcal B)|$ denotes the number of distinct vertex labels and $|E(\mathcal T)|$ denotes the number of edges in the subtree $\mathcal B$. 

\begin{lem}
    Let $\mathcal T\in \mathcal T_{i, t, m}(\mathcal B)$ where \[
\mathcal{T}_{i,t,m}(\mathcal{B})
    = \bigl\{\, \text{labeled monomial trees whose associated bad subtree is exactly } \mathcal{B} \,\bigr\}.
\]
Then recall by Lemma~\ref{prop:bad_subtree_properties}, the boundary vertices are paired by labels, giving pairs \( (v,w)\in\partial \mathcal{B}\). Then we have that, \[
\operatorname{Val}(\mathcal T)
= \operatorname{Val}_1(\mathcal B)\cdot
  \Bigg(\prod_{(v, w)\in\partial\mathcal B}\operatorname{Val}_2(\mathcal T_v\star \mathcal T_w)\Bigg)
\]
Where on the RHS we take a product over pairs $(v, w)$ and the notation $\mathcal T_v$ denotes the subtree of $\mathcal T$ rooted at $v$, so $\mathcal T_v\star \mathcal T_w$ is the subtrees rooted at $v$ and $w$ glued at their root. 
\end{lem}

\begin{proof}
    The proof follows by the definitions of $\operatorname{Val}(\mathcal T), \operatorname{Val}_1(\mathcal B), \operatorname{Val}_2(\mathcal S)$ and is basically Lemma \ref{lem:factor_val_forest}.
\end{proof}

\begin{prop} \label{prop:bad_subtrees}
Let \( \mathcal{B} \) be a possible labeled bad subtree. Then
\[
\sum_{\mathcal{T}\in \mathcal{T}_{i,t,m}(\mathcal{B})} \mathbb E\operatorname{Val}(\mathcal{T})
    =  \mathbb E\operatorname{Val}_1(\mathcal{B})
          \cdot \left[\prod_{(v,w)\in \partial \mathcal{B}}
              \mathrm{Wick}\left( T(x_{i_v}^{s_v}) \star T(x_{i_w}^{s_w}) \right)+O(N^{-1+10/K_1})\right]
\]

Here, boundary vertices are paired by labels, giving pairs \( (v,w)\in\partial \mathcal{B} \).  
- For a boundary vertex \( v \), the number \( s_v \) denotes its generation, and \( T(x_{i_v}^{s_v}) \in \mathcal{A} \) is the moment-algebra element corresponding to the iterate \( x_{i_v}^{s_v} \), where $i_v$ is the label of vertex $v$ in $\mathcal B$.
\end{prop}

\begin{proof}
Every \( \mathcal T \in \mathcal{T}_{i,t,m}(\mathcal{B}) \) can be written uniquely as
\[
\mathcal T \;=\; \mathcal B \;\cup\; \bigcup_{(v, w)\in\partial\mathcal B} \mathcal T_v\star \mathcal T_w,
\]
where \(\mathcal T_v\) is the (possibly trivial) good subtree attached at the boundary vertex \(v\), and $\mathcal T_v\star \mathcal T_w$ denotes the two subtrees of a pair of boundary vertices $v$ and $w$ glued at the root. The excess $\Delta(\mathcal T_v\star \mathcal T_w)=0$ is given by Proposition \ref{prop:bad_subtree_properties} part (3). Furthermore, since $\mathcal T_v\star \mathcal T_w$ contains only good vertices, it is disjoint from the other components of $\mathcal T$.

For such a decomposition,
\[
\operatorname{Val}(\mathcal T)
= \operatorname{Val}_1(\mathcal B)\cdot
  \Bigg(\prod_{(v, w)\in\partial\mathcal B}\operatorname{Val}_2(\mathcal T_v\star \mathcal T_w)\Bigg)
\]
The remainder of the proof essentially follows by grouping labeled trees $\mathcal T_v\star \mathcal T_w$ into isomorphism classes and recalling the definition of the Wick pairing product. By Lemma \ref{lem:delta_0_perfect_pairing} each labeled subtree $\mathcal T_v\star \mathcal T_w$ is perfectly Wicked paired with root label $i_v=i_w$. If vertices $v, w$ are in generations $s_v, s_w$, then $T(x_{i_v}^{s_v})$ and $T(x_{i_w}^{s_w})$ denote precisely the elements of the moment algebra corresponding to the weighted sum of possible unlabeled subtrees from $v$ and $w$. By Proposition \ref{prop:delta_zero_covariance}, we have the equality \[\sum_{\mathcal T_v, \mathcal T_w}\mathbb E\operatorname{Val}_2(\mathcal T_v\star \mathcal T_w)=\operatorname{Wick}(T(x_{i_v}^{s_v})\star T(x_{i_w}^{s_w}))+O(N^{-{1+10/K_1}}))\] 
The same argument will show that,
\[
\Bigg(\sum_{\mathcal T_v, \mathcal T_w|(v, w)\in \mathcal \partial \mathcal B}\prod_{(v, w)\in\partial\mathcal B}\operatorname{Val}_2(\mathcal T_v\star \mathcal T_w)\Bigg)=\prod_{(v, w)\in \partial \mathcal B}\mathrm{Wick}\!\left( T(x_i^{s_v}) \star T(x_i^{s_w})\right)+O(N^{-{1+10/K_1}})
\]
Since we can view the collection of tree pairs over $\partial \mathcal B$ as a single tree with a root that does not contribute to its value.

Summing \(\operatorname{Val}(\mathcal T)\) over all \(\mathcal T\in\mathcal T_{i,t,m}(\mathcal B)\), we get
\[
\sum_{\mathcal T\in\mathcal T_{i,t,m}(\mathcal B)} \mathbb E\operatorname{Val}(\mathcal T)
=
\mathbb E\operatorname{Val}(\mathcal B)
\cdot
\left[\prod_{(v,w)\in\partial\mathcal B}
\mathrm{Wick}\!\left( T(x_i^{s_v}) \star T(x_i^{s_w})\right)+O(N^{-{1+10/K_1}}) \right]
\]
which is exactly the claimed identity.
\end{proof}

The Proposition \ref{prop:bad_subtrees} allows us to bound the contribution from trees with excess $\Delta>0$.

\begin{proof}[Proof of Theorem \ref{thm:State_evolution_approximation_polynomial}]
    Let \( T(x_i^t) \) denote the collection unlabeled monomial trees (with monomial degrees \( d_v \leq D \)) of depth \( t \) corresponding to the iterate \( x_i^t \). Then by Lemma~\ref{lem:tree_expansion}, we have
\[
\mathbb{E}\left[(x_i^t)^{m}\right] = \sum_{T\in T(x_i^t)} \sum_{\phi: V(T) \to [N]}\mathbb{E} \operatorname{Val}(\mathcal{T}_\phi),
\]
where \( \phi \) is a labeling of the tree \( T \), and \( \mathcal{T}_\phi \) is the corresponding labeled monomial tree.

We group the sum over labelings according to the isomorphism class \( \mathcal{I} \) of the labeled tree \( \mathcal{I}_\phi \), and classify these by their excess \( \Delta(\mathcal{I}) \). Let us first consider the contribution from isomorphism classes \( \mathcal{I} \) with \( \Delta(\mathcal{I}) = 0 \). By Lemma~\ref{lem:delta_0_perfect_pairing}, these correspond exactly to Wick pairings, and the number of such isomorphism classes is \( \mathrm{Wick}(T) \). Each such class contributes:
\[
\sum_{\psi: V(\mathcal{L}) \to [N] \text{ injective}} \prod_{v \in V(T_\mathcal{I})} c_v \cdot \prod_{(u, v) \in E(T_\mathcal{I})} A_{\psi(u)\psi(v)} \cdot \prod_{\ell \in L(T_\mathcal{I})} x^0_\ell.
\]

Each distinct edge contributes \( \mathbb{E}[A_{ij}^2] = 1/N \), so the edge product yields \( N^{-|E(\mathcal{L})|/2} \). By Proposition \ref{prop:sum_delta_0_accurate} the total contribution from all \( \Delta = 0 \) terms is:
\[
\mathbb E\left[(y_i^t)^{m}\right] + O(N^{-1 + 10/K_1})
\]
or $0$ when $m$ is odd. Now consider terms with \( \Delta(\mathcal{T}_\phi) > 0 \). For each such tree $\mathcal T_\phi$ let $\mathcal B_\phi$ denote its bad subtree. By Proposition \ref{prop:bad_subtrees},

\[
\sum_{\mathcal T\in\mathcal T_{i,t,m}(\mathcal B)} \operatorname{Val}(\mathcal T)
=
 \mathbb E\operatorname{Val}(\mathcal B)
\cdot
\left[\prod_{(v,w)\in\partial\mathcal B}
\mathrm{Wick}\left( T(x_i^{s_v}) \star T(x_i^{s_w}) \right)+O(N^{-1+10/K_1})\right]
\]

By the Cauchy Schwarz inequality: \[\operatorname{Wick}(T(x_i^{s_v})\star T(x_i)^{s_w})\leq \operatorname{Wick}(T(x_i^{s_v}))^{1/2}\cdot  \operatorname{Wick}(T(x_i^{s_w}))^{1/2}\leq \mathbb E[(x_i^{s_w})^2]^{1/2}\mathbb E[(x_i^{s_v})^2]^{1/2}\leq M\]
Furthermore, the number of vertices in $\partial \mathcal B$ is at most $C_3Dt\Delta(\mathcal B)$ where $C_3$ is a absolute constant. This follows because there are at most $C_1\Delta(\mathcal B)$ bad vertices hence at most $C_1t\Delta(\mathcal B)$ branch vertices. Recall as in Theorem \ref{thm:universality_algebra_elements},
\[
\mathbb{E}\left[\prod_{(u,v) \in E(\mathcal{B})} A_{i_ui_v}\right]
\leq C_2^{6\Delta} \cdot (6\Delta)^{3\Delta} \cdot N^{-|E(\mathcal{B})|/2}.
\]
Where $i_u, i_v$ denote the labels of vertices $u$ and $v$. Recall the assumption $\mathbb E|x_i^0|^p\leq (Kp)^{p/2}$
so:
\[
\mathbb E\left| \prod_{\ell \in L(\mathcal{B}\backslash \partial \mathcal B)} x^0_\ell \right| \leq  (K\Delta)^{(K\Delta)}\cdot \tau_0^{2k},
\]
where \( k \) is the number of leaf pairs, and the excess comes from leaves with multiplicity \( >2 \). Also, $c_{s, d}<M$ by assumption. This yields,
\[
\mathbb E\operatorname{Val}(\mathcal B)
\cdot
\prod_{(v,w)\in\partial\mathcal B}
\mathrm{Wick}\!\left( T(x_i^{s_v}) \star T(x_i^{s_w}) \right)\leq (2M)^{t\Delta(\mathcal  B)}\cdot (CD)^{CD}\cdot M^{C_3Dt\Delta(\mathcal B)}\cdot N^{-\Delta(\mathcal B)}
\]

By Proposition \ref{prop:num_bad_subtrees} we have that the number of possible subtrees $\mathcal B$ with excess $\Delta$ is at most $N^{C_\Delta/K_1}$. Therefore when we sum the tree with positive excess $\mathcal T_{i,t,m,>0}$ we get,
\begin{align*}
\sum_{\mathcal T\in\mathcal T_{i,t,m,>0}} \mathbb E\operatorname{Val}(\mathcal T)&=\sum_{\mathcal B|\Delta(\mathcal B)=\Delta}\sum_{\mathcal T\in\mathcal T_{i,t,m}(\mathcal B)} \mathbb E\operatorname{Val}(\mathcal T)\leq\\ &\sum_{\Delta\geq 1/2}N^{C_4\Delta/K_1}\cdot M^{t\Delta}\cdot (CD)^{CD}\cdot M^{C_3Dt\Delta}\cdot N^{-\Delta}
\end{align*}
Now the RHS is $O(N^{-1/2+B/K_1})$ since $t\leq \log N/(K_1D\max\{\log M, \log D\})$ for some absolute constant $B$. By Proposition \ref{prop:bad_subtrees} when we sum over trees with excess $0$, $\mathcal T_{i,t,m,0}$, we get,
\[
\sum_{\mathcal T\in\mathcal T_{i,t,m}}\mathbb E\mathrm{Val}(\mathcal T)=\mathbb E(y_i^t)^{2m}+O(N^{-1/2+B/K_1})
\]
Hence the Theorem is proven.
\end{proof}

\begin{proof}[Proof of Theorem \ref{thm:convergence_amp_probability}]
    By Theorem \ref{thm:State_evolution_approximation_polynomial} it remains to prove that $\mathrm{Var}(M_{t,m}^{(N)})\to0$ as $N\to\infty$. Write $X_i^{(N)}:=(x_i^t)^m$ and $\mu_{t,m}:=\mathbb E[y_t^m]$.
By exchangeability,
\[
\mathrm{Var}(M_{t,m}^{(N)})
 = \frac1N \mathrm{Var}(X_1^{(N)}) + \frac{N-1}{N}\mathrm{Cov}(X_1^{(N)},X_2^{(N)}).
\]
By Theorem~\ref{thm:State_evolution_approximation_polynomial}
with $m$ replaced by $2m$ we have
\[
\mathbb E\big[(X_1^{(N)})^2\big] = \mathbb E\big[y_t^{2m}\big] + O\big(N^{-1/2+B/K_1}\big),
\]
so in particular $\sup_N\mathbb E[(X_1^{(N)})^2]<\infty$ and hence $\mathrm{Var}(X_1)=O(1)$. Therefore
\[
\frac1N\mathrm{Var}(X_1^{(N)}) = O\Big(\frac1N\Big)\to 0.
\]
It remains to bound the covariance term $\mathrm{Cov}(X_1^{(N)},X_2^{(N)})$.

For this, consider the mixed moment
\[
\mathbb E\big[X_1^{(N)}X_2^{(N)}\big]
  = \mathbb E\big[(x_1^t)^m(x_2^t)^m\big].
\]
As in the proof of Theorem~\ref{thm:State_evolution_approximation_polynomial},
this expands as a sum over pairs of monomial trees with two distinguished
roots labelled $1$ and $2$, together with vertex labelings. The same grouping
by isomorphism classes and the same excess parameter $\Delta$ apply. An
isomorphism class may have $\Delta=0$ only when each component tree has
$\Delta=0$ individually and the two components share no labels; in this case
the Wick functional factorizes and gives
\[
\mathbb E[X_1^{(N)}X_2^{(N)}]
 = \mathrm{Wick}(T(x_1^t)^m)\cdot\mathrm{Wick}(T(x_2^t)^m)
 + O(N^{-1+10/K_1})
 = \mu_{t,m}^2 + O(N^{-1+10/K_1}).
\]

To control the contribution of $\Delta>0$ classes, we use the same
bad-subtree decomposition as in Definition~\ref{defn:bad_subtree}.
In the tree expansion of $\mathbb E[X_1^{(N)}X_2^{(N)}]$, each term is indexed by a
pair of labeled monomial trees $(\mathcal T_1,\mathcal T_2)$ corresponding to $(x_1^t)^m$ and $(x_2^t)^m$,
together with a labeling of their vertices. For each such pair we form a single
monomial tree $\widetilde{\mathcal T}$ by gluing $\mathcal T_1$ and $\mathcal T_2$ under a new root of
degree $2$ (with polynomial weight $1$). The value of the glued tree is then
exactly the product of the values of the two original trees, so the sum over
pairs $(\mathcal T_1,\mathcal T_2)$ may be identified with a sum over such glued trees
$\widetilde{\mathcal T}$.

We define the bad subtree of a pair $(\mathcal T_1,\mathcal T_2)$ to be the bad subtree of its
glued tree $\widetilde{\mathcal T}$. Since each glued tree has the same depth, degrees,
and polynomial weights as in Theorem~\ref{thm:State_evolution_approximation_polynomial},
the estimates there apply similarly to this family. In particular, the total
contribution of all $\Delta>0$ glued trees (and hence of all $\Delta>0$
pairs $(\mathcal T_1,\mathcal T_2)$) is
\[
O(N^{-1/2+B/K_1}),
\]
with constants depending only on $t,m$.

On the other hand, Theorem~\ref{thm:State_evolution_approximation_polynomial} gives
\[
\mathbb E\big[X_1^{(N)}\big] = \mu_{t,m} + O\big(N^{-1/2+B/K_1}\big),
\]
and the same for $X_2^{(N)}$, so
\[
\mathbb E\big[X_1^{(N)}\big]\mathbb E\big[X_2^{(N)}\big]
  = \mu_{t,m}^2 + O\big(N^{-1/2+B/K_1}\big).
\]
Subtracting, we obtain
\[
\mathrm{Cov}(X_1^{(N)},X_2^{(N)})
  = O\big(N^{-1/2+B/K_1}\big).
\]
Thus
\[
\mathrm{Var}(M_{t,m}^{(N)})
 = O\Big(\frac1N\Big) + O\big(N^{-1/2+B/K_1}\big)\xrightarrow[N\to\infty]{}0,
\]
since $K_1$ is fixed and large. By Chebyshev's inequality this implies
$M_{t,m}^{(N)}-\mathbb E[M_{t,m}^{(N)}]\to 0$ in probability, and
Theorem~\ref{thm:State_evolution_approximation_polynomial} gives
$\mathbb E[M_{t,m}^{(N)}]\to\mu_{t,m}=\mathbb E[y_t^m]$, so
$M_{t,m}^{(N)}\to \mathbb E[y_t^m]$ in probability as claimed.
\end{proof}

\subsection{Counter Example}
\label{sec:counter_example}
We now show that the growth-rate condition in Theorem
\ref{thm:State_evolution_approximation_polynomial} is sharp.  
The theorem shows that state evolution remains valid for a fixed degree $D$ and for any
$t = O(\log N)$, but may fail when $D$ grows with $N$. (Note here that the implicit constant of $O(\cdot)$ depends on $D$)
We construct an explicit AMP instance (satisfying all hypotheses
except the growth bound) for which a single tree diagram contributes an exponentially
larger amount than the state evolution prediction once
\[
t \;\gtrsim\; \frac{\log N}{D\log D}.
\]

\begin{prop}
There exists an AMP instance satisfying all assumptions of
Theorem~\ref{thm:State_evolution_approximation_polynomial}
except the iteration bound, for which state evolution fails
when
\(
\frac{\log N}{D\log D}\lesssim t\lesssim\frac{\log N}{\log D}.
\)
In this example $D$ should grow with $N$.
\end{prop}
\begin{proof}
Assume $D$ is even and let
\[
f_s(x) = \frac{x^D}{\sqrt{(2D-1)!!}}, \qquad 0\le s\le t.
\]
For $Z\sim\mathcal N(0,1)$ we have
$\mathbb E[f_s(Z)^2]=1$, and hence the state-evolution recursion gives
$\tau_s^2=1$ for all $s$.  
Initialize $x^0=\mathbf 1_N$.  
Thus the uniform boundedness assumptions on the coefficients and variances are satisfied. The off-diagonal entries of $\sqrt N\cdot A$ are IID Gaussian $A_{ij}\sim\mathcal N(0,1)$.

We consider the iterate $(x_i^{t})^m$, whose tree expansion consists of one unlabeled tree $T$, whose root has degree $m$ and all other vertices have degree $D$. Then $T$ satisfies:
\begin{itemize}
    \item the tree has depth $t$
    \item the root (generation $0$) has degree $m$
    \item every internal vertex has degree $D+1$ (one parent edge, $D$ child edges)
    \item the last generation $t$ (depth $t$ and furthest from root) consists of exactly \(mD^{t-1}\) leaves.
\end{itemize}

We consider labelings of $T$ in which:
\begin{itemize}
    \item[(1)] all vertices in generations $0,1,\dots,t-1$ share a common label  
   (so generation $s$ has label $\ell_s$)
   \item[(2)] the leaves at depth $t$ (generation $t$) are paired arbitrarily and each pair receives a single label.
\end{itemize}

The first condition gives exactly \(N^{t-1}\) choices for the labels \(\ell_1,\dots,\ell_{t-1}\). For the leaves, the number of possible pairings is \(
(mD^{t-1}-1)!!,
\) since all vertices in generation $t$ have the same label
we can choose any pairing of the leaf vertices. The number of ways to assign labels to the pairs is the falling factorial
\(
(N)_{mD^{t-1}/2}
\sim
N^{mD^{t-1}/2}
\) since we assume $D^{t-1}\lesssim N^c$ (since $t\lesssim\log N/\log D$) where $c<1/2$ is a constant.

Thus the number of labelings of this type is
\(
N^{t-1} \cdot (mD^{t-1}-1)!! \cdot N^{mD^{t-1}/2}
\) up to the constant factor from the approximation for the falling factorial. Each edge contributes a factor \(N^{-1/2}\) from the scaling of $A$.   
The tree has  
\(
m + mD + mD^2 + \cdots + mD^{t-1}
\)
edges. The AMP normalization gives a factor
\(
N^{-\#E} \geq  N^{-(m + mD + mD^2 + \cdots + mD^{t-1})/2}.
\)
The normalization of the coefficients in $f_s$ gives a factor of $\frac{1}{\sqrt{(2D-1)!!}}$ for each non-leaf vertices for a total factor \[[(2D-1)!!]^{-(m+mD+\dots +m D^{t-2})/2}\geq (2D)^{-4mD^{t-1}}\]
Multiplying by the number of labelings, we obtain,
\begin{align*}
\text{Contribution}(T)
&\gtrsim
N^{t-1} \cdot (mD^{t-1}-1)!! \cdot N^{-(m + mD + mD^2 + \cdots + mD^{t-1})/2} \cdot N^{mD^{t-1}/2}\cdot (2D)^{-4mD^{t-1}}\\
&\gtrsim
(mD^{t-1}-1)!!\cdot N^{-mD^{t-2}}\cdot (2D)^{-4mD^{t-1}}
\end{align*}
Using Stirling approximation,
\[
(mD^{t-1}-1)!! \sim (mD^{t-1})^{mD^{t-1}/2} \;=\; \exp\!\left( \frac{mD^{t-1}}{2}\log ({m D^{t-1}} )\right)
 = \exp\!\left( \frac{mD^{t-1}}{2}\cdot (t-1)\log (mD) \right)
\]

Hence
\begin{align}
\log\big(\text{Contribution}(T)\big)
&\gtrsim\\ \label{eq:exponent2} mD^{t-1}(t-1)\log(mD)&-\log N\cdot (mD^{t-2})-\log(2D)(4mD^{t-1})\gtrsim
\\ \label{eq:exponent}
mD^{t-1}(t-9)\log(mD)&-\log N\cdot (mD^{t-2})
\end{align}

Where $(t-9)$ comes from absorbing the third factor of \eqref{eq:exponent2}. State evolution predicts a contribution of order \(1\).
Therefore, SE fails exactly when the exponent on the RHS of \eqref{eq:exponent} grows with $N$. Assume $m=2$ or any small constant.

The first term dominates as soon as
\[
D^{t-1}  t\log D \gg \log N \cdot D^{t-2}
\]

Rearranging gives
\[
t \;\gtrsim\; \frac{\log N}{D\log D},
\]
which is precisely the boundary claimed.

Thus for any \(t\) in this regime the contribution of a single tree diagram diverges and the state evolution approximation breaks down.
\end{proof}

\section{AMP with a Spiked Matrix Model}
\label{sec:spiked}

In this section, we give a brief overview of the Moment Algebra for a spiked model. We believe a similar argument should work for the spiked model, but we do not provide this argument in this work.

Consider AMP applied to a symmetric rank \( 1 \) spiked Wigner model:
\begin{align}\label{eq:A_spiked_rank_r}
A = \frac{\lambda}{N} v_* v_*^\top + W,
\end{align}
where \( W\in \mathbb R^N \) satisfies the same assumptions as before (symmetric, sub-Gaussian entries, variance \(1/N\)), and \( v_* \in \mathbb{R}^N \) satisfies \( \| v_* \|^2 = N \). For simplicity assume the initial condition is $x^0=1_N\in \mathbb R^N$ i.e. the vector of all $1$'s in $\mathbb R^N$. Then we have that $\sigma_0=\sigma_1=1$ and $\mu_1=\frac{1}{N}\langle x^0, v^*\rangle=\sum_i\lambda v_{*, i}/N $. Let $V=\lambda v_*v_*^\top/N$ 

The main difference between the spiked and unspiked models is that one should keep track of the signal, which can be done in the following way. Each edge in a tree corresponding to an AMP term is assigned a color:
\begin{itemize}
\item \textbf{Yellow edges:} correspond to \( V \) terms.
\item \textbf{Blue edges:} correspond to \( W \) terms.
\end{itemize}

The colored Wick moment algebra is the natural generalization of the unspiked moment algebra. The main idea is to track the number of ways to pair the blue edges and then multiply by the deterministic contribution from the yellow edges.

\begin{defn}[Colored Tree]
A \textbf{colored tree} \( \mathcal T \) is a tree arising from the expansion of AMP iterates, where each edge is assigned a color: yellow (signal) or blue (noise). The total contribution of a tree is calculated as a product over all edges.
\[
\prod_{\text{yellow edges: (i, j)}} \frac{\lambda}{N}v_{*,i}v_{*,j}\;   \prod_{\text{blue edges: (i, j)}} W_{ij} 
\]
where the product runs over edges, and the \( u,v \) are the labels of the incident vertices.
\end{defn}

The following Spiked Wick product on unlabeled colored trees will allow one to compute the expected contribution of the colored tree to the Spiked AMP state evolution moment formulas.

\begin{defn}[Spiked Wick Product]\label{defn:colored_wick_product}
Let \( T \) be a rooted, colored, unlabeled monomial tree (with the root labeled by \( i \in [N] \) and all other vertices unlabeled), where each edge is either:
\begin{itemize}
    \item blue (representing a noise edge \( W_{ij} \)), or
    \item yellow (representing a signal edge from the rank-one spike \( \frac{\lambda}{N} v_* v_*^\top \)).
\end{itemize}

We define the \emph{spiked Wick product} \( \mathrm{Wick}_{\mathcal{S}}({T}) \in \mathbb{R} \) recursively as follows:
\begin{itemize}
    \item \textbf{Base case:} If \( {T} \) is a leaf, then
    \[
    \mathrm{Wick}_{\mathcal{S}}({T}) := 1.
    \]

    \item \textbf{Recursive step:} Suppose the root of \( {T} \) has children \( T_1, \dots, T_d \), and let \( B \subseteq [d] \) denote the indices of the children attached to the root by blue edges. Then:
    \[
    \mathrm{Wick}_{\mathcal{S}}({T}) :=
    \sum_{\text{pairings } \pi \text{ of } B}
    \left(
        \prod_{\{a, b\} \in \pi} \mathrm{Wick}_{\mathcal{S}}(T_{a} \star T_{b})
    \right)
    \cdot
    \left(
        \prod_{k \in [d] \setminus B} \frac{\lambda}{N} \cdot \sum_{j=1}^N v_{*, i}v_{*,j} \cdot \mathrm{Wick}_{\mathcal{S}}(T_k^{(j)})
    \right),
    \]
    where:
    \begin{itemize}
        \item \( T_{a} \star T_{b} \) denotes the tree obtained by gluing \( T_{a} \) and \( T_{b} \) at the root,
        \item \( T_k^{(j)} \) denotes the tree \( T_k \) with the root labeled by \( j \in [N] \),
    \end{itemize}
\end{itemize}
\end{defn}

\bibliographystyle{alpha}
\bibliography{amp}

\end{document}